\documentclass[11pt]{article}
\usepackage[tbtags]{amsmath}
\usepackage{cases}
\usepackage{mathrsfs}
\usepackage{amsfonts}
\usepackage{authblk}
\usepackage{amsmath,amsthm,amssymb}
\usepackage{cite}

\usepackage{color}

\theoremstyle{plain}
\newtheorem{theorem}{Theorem}[section]
\newtheorem{lemma}{Lemma}[section]
\newtheorem{remark}{Remark}[section]
\newtheorem{proposition}{Proposition}[section]

\newif \ifLastSection \LastSectionfalse

\numberwithin{equation}{section}

\allowdisplaybreaks

\begin{document}

\title{{\bf {Long-time behavior of solutions to the $M_1$ model with boundray effect}}}

\author[1]{Nangao Zhang}
\author[2]{Changjiang Zhu\thanks{Corresponding author. \authorcr Email addresses: mazhangngmath@mail.scut.edu.cn
 (zhang), machjzhu@scut.edu.cn (zhu)
.}}
\affil[1,2]{ \normalsize  School of Mathematics, South China University of Technology, Guangzhou 510641, P.R. China}

\date{}

\maketitle

\textbf{{\bf Abstract:}} In this paper, we are concerned with the asymptotic behavior of solutions of $M_1$ model on quadrant $(x,t) \in \mathbb{R}^{+} \times \mathbb{R}^{+}$.
From this model, combined with damped compressible Euler equations, a more general system is introduced.
We show that the solutions to the initial boundary value problem of this system globally exist and tend time-asymptotically to the corresponding nonlinear parabolic equation governed by the related Darcy's law.
Compared with previous results on compressible Euler equations with damping obtained by Nishihara and Yang in \cite{Nishihara-Yang1999}, and Marcati, Mei and Rubino in \cite{Marcati-Mei-Rubino2005}, the better convergence rates are obtained. The approach adopted is based on the technical time-weighted energy estimates together with the Green's function method.

\bigbreak \textbf{{\bf Key Words}:} $M_1$ model, nonlinear parabolic equation, time-weighted energy estimates, asymptotic behavior.

\bigbreak \textbf{{\bf AMS Subject Classification}:} 85A25, 35L65, 35B40.

\section{Introduction}\label{S1}

Radiation transport is involved in many fields, such as climate, medicine and so on. In recent years, a great effort has been made to develop approximate models that solve the radiative transfer equations at a low cost. In this context, the $M_{1}$ model is an interesting choice (cf. \cite{Berthon-Charrier2003, Berthon-Dubois2010}). Here we just consider the scattering part and we omit the role played by the temperature, then the corresponding simplified model reads
(cf. \cite{Berthon-Charrier-Dubroca2007, Goudon-Lin2013}):
\begin{equation}\label{1.1a}
\left\{\begin{array}{l}
\partial_{t}\rho+c\nabla \cdot (\rho u)=0,\\[2mm]
  \partial_{t}(\rho u)+ c\nabla \cdot P(\rho, u)=-c\sigma\rho u.
 \end{array}
        \right.
\end{equation}
Here, the unknown function $\rho=\rho(x,t) \geq 0$ and $u=u(x,t)  \in \mathbb{R}^{n} ~ (1\leq n\leq3)$ denote respectively the radiative energy and normalized radiative flux. The positive constants $c$ and $\sigma$ denote the speed of the light and the opacity.
Concerning the radiative pressure $P(\rho, u)$, it is given by
\begin{equation}\label{1.2a}
P(\rho, u)=\frac{1}{2}\left((1-\chi(u)) \mathbb{I}_{n}+(3 \chi(u)-1) \frac{u \otimes u}{|u|^{2}}\right) \rho,
\end{equation}
with
\begin{equation}\label{1.3a}
  \chi(u)=\frac{3+4 |u|^{2}}{5+2 \sqrt{4-3|u|^{2}}},
\end{equation}
 where $\mathbb{I}_{n}$ is the identity matrix of order $n$ and $\left|u\right| \leq 1$.

In this paper, we shall restrict ourselves to the one-dimensional case. We set $c=1$ without loss of generality, then \eqref{1.1a} can be rewritten as
\begin{equation}\label{1.4a}
\left\{\begin{array}{l}
 \rho_{t} +\left(\rho u \right)_{x}=0,\\[2mm]
\displaystyle  \left(\rho u \right)_{t}+\left(\frac{\rho}{3}\right)_{x}+\left(\frac{2\rho u^{2}}{2+\sqrt{4-3u^{2}}} \right)_{x}=-\sigma \rho u.
 \end{array}
        \right.
\end{equation}
We are interested in the large time behavior of solutions to \eqref{1.4a} on quadrant. Suppose that $\rho \geq C > 0$, then it is more convenient to use the Lagrangian coordinates to explore this system. We consider the coordinate transformation as follows:
\begin{equation}\notag
 x \Rightarrow \int_{(0,0)}^{(x, t)} \rho(y, s) \mathrm{d} y-(\rho u)(y, s) \mathrm{d} s, \quad t \Rightarrow \tau ,
\end{equation}
 and we still denote the Lagrangian coordinates by $(x,t)$ for simplicity.

 Let $v=\frac{1}{\rho}$, then the system \eqref{1.4a} can be transformed as the following form
\begin{equation}\label{1.5a}
 \left\{\begin{array}{l}
 v_{t}-u_{x}=0,\\[2mm]
\displaystyle  u_{t}+\left(\frac{1}{3v} \right)_{x}-\left(\frac{u^{2}\sqrt{4-3u^{2}}}{v(2+\sqrt{4-3u^{2}})} \right)_{x}=- \sigma u, \quad (x,t) \in \mathbb{R}^{+} \times \mathbb{R}^{+},
 \end{array}
        \right.
\end{equation}
with the following initial data
\begin{equation}\label{1.6a}
 (v,u)|_{t=0}=(v_0,u_0)(x) \rightarrow (v_{+}, u_{+}),  \quad v_{+}>0£¬ \quad \mbox{as} \quad x \rightarrow \infty ,
\end{equation}
and the Dirichlet boundary condition
\begin{equation}\label{1.7a}
 u(0,t)=0.	
 \end{equation}

Now let's briefly review the previous work on $M_{1}$ model and introduce the main concerns of this paper. For the Cauchy problem, the global existence of smooth solutions with small initial data has been studied in \cite{Goudon-Lin2013}, Zhang and Zhu in \cite{Zhang-Zhu2021} showed that the solution of (1.5) tended time-asymptotically to the nonlinear diffusion waves, and obtained its optimal convergence rate.
However, the large time behavior of solutions to $M_{1}$ model on quadrant is never dealt with as we know.
Therefore, the target in this paper is to study the large time behaviors of smooth solutions to \eqref{1.5a}--\eqref{1.7a}.
It is noted that in physics, the damping effects make that the dynamical system possesses diffusion phenomena like nonlinear diffusion equations, so we are mainly focus on the nonlinear diffusive phenomena of this system.

Noting the relationship between \eqref{1.5a} and the compressible Euler equations with damping, we hope to obtain more general results involving the two systems. So in this paper, we tend to consider the following more general system
\begin{equation}\label{1.1}
 \left\{\begin{array}{l}
 v_{t}-u_{x}=0,\\[2mm]
  u_{t}+p(v)_{x}=-\alpha u+(g(u)f(v))_{x}, \quad (x,t) \in \mathbb{R}^{+} \times \mathbb{R}^{+},
 \end{array}
        \right.
\end{equation}
with initial data
\begin{equation}\label{1.2}
  (v,u)|_{t=0}=(v_0,u_0)(x) \rightarrow (v_{+}, u_{+}), \quad \mbox{as} \quad x \rightarrow \infty ,
\end{equation}
and with the Dirichlet boundary condition
\begin{equation}\label{1.3}
u|_{x=0}=0.	
\end{equation}
 Here $u=u(x,t) ~\mbox{and}~ v=v(x,t)>0:\mathbb{R}^{+} \times \mathbb{R}^{+} \rightarrow \mathbb{R}$ are unknown, $p$ is a smooth function of $v$ with $p>0$, $g$ and $f$ are smooth function of $u$ and $v$, respectively, physical coefficients $\alpha>0$. $v_{+}>0$ and $u_{+}$ are constants.

 Let's review some relevant mathematical investigations. When $(g(u)f(v))_{x} \equiv 0$, the system \eqref{1.1}  reduces to compressible Euler equations with linear damping
\begin{equation}\label{1.4}
 \left\{\begin{array}{l}
 v_{t}-u_{x}=0,\\[2mm]
  u_{t}+p(v)_{x}=-\alpha u.
 \end{array}
        \right.
\end{equation}
For the Cauchy problem to \eqref{1.4}, there is a huge literature on the investigations of global existence and large time behaviors of smooth solutions (see \cite{Nishihara1997, Wang-Yang2001, Zhao2000} and references therein).
Hsiao and Liu in \cite{Hsiao-Liu1992} firstly obtained the solutions $(v,u)(x,t)$ to the corresponding Cauchy problem of \eqref{1.4} tend time-asymptotically to the nonlinear self-similar diffusion wave solutions $(\bar{\bar{v}},\bar{\bar{u}})(x,t)$ of the porous media equations
\begin{equation}\notag
 \left\{\begin{array}{l}
\bar{\bar{v}}_{t}-\bar{\bar{u}}_{x}=0,\\[2mm]
p(\bar{\bar{v}})_{x}=- \alpha\bar{\bar{u}},\quad (x,t) \in \mathbb{R} \times \mathbb{R}^{+},
 \end{array}
        \right.
\end{equation}
in the sense $\left\|(v-\bar{\bar{v}}, u-\bar{\bar{u}})(t) \right\|_{L^{\infty}} \leq C(t^{-\frac{1}{2}},t^{-\frac{1}{2}})$.
Then, by taking more detailed energy estimates, Nishihara in \cite{Nishihara1996} successfully improved the convergence rates
as $\left\|(v-\bar{\bar{v}}, u-\bar{\bar{u}})(t) \right\|_{L^{\infty}} \leq C(t^{-\frac{3}{4}},t^{-\frac{5}{4}})$.
Furthermore, by constructing an approximate Green function with the energy method together, Nishihara, Wang and Yang in \cite{Nishihara-Wang-Yang2000} completely improved the convergence rates as $\left\|(v-\bar{\bar{v}}, u-\bar{\bar{u}})(t) \right\|_{L^{\infty}} \leq C(t^{-1},t^{-\frac{3}{2}})$, which is optimal in the sense comparing with the heat equation.
These conclusions require that both the initial disturbance and the wave strength around a particular diffusion wave are suitably small, some of these restrictions were later partially relaxed by Zhao in \cite{Zhao2001}.
For the Cauchy problem to \eqref{1.4} with nonlinear damping or vacuum, and so on, we can refer to these interesting works \cite{Hsiao-Liu1993, Hsiao-Luo1996, Huang-Pan2003, Huang-Pan2006, Mei2009,Mei2010, Nishihara2003, Zhu-Jiang2006} and references therein.

 For the initial-boundary value problem (IBVP) on the quarter plane $\mathbb{R}^{+} \times \mathbb{R}^{+}$ to \eqref{1.4}, the global existence and the asymptotic behavior of the solution have been investigated by several authors (see \cite{Ma-Mei2010, Marcati-Mei2000} and references therein). Nishihara and Yang in \cite{Nishihara-Yang1999} considered the asymptotic behavior of solution to \eqref{1.4} with the Dirichlet boundary condition \eqref{1.3}, and they got the convergence rates in form of $\left\|(v-\tilde{v}, u-\tilde{u})(t) \right\|_{L^{\infty}} \leq C(t^{-\frac{3}{4}},t^{-\frac{5}{4}})$ by perturbing the initial value around the linear diffusion waves $(\tilde{v},\tilde{u})(x,t)$ which satisfies
\begin{equation}\notag
\left\{\begin{array}{l}
 \tilde{v}_{t}-\tilde{u}_{x}=0,\\[2mm]
  p^\prime(v_{+})\tilde{v}_{x}=-\alpha \tilde{u},\quad (x,t) \in \mathbb{R}^{+}\times \mathbb{R}^{+},\\[2mm]
 \tilde{u}|_{x=0}=0,\quad (\tilde{v},\tilde{u})|_{x=\infty}=(v_{+},0),
 \end{array}
   \right.	
\end{equation}
 when the initial perturbation belonged to $H^{3}\times H^{2}$. Observing that the decay rates can be better by eliminate the slower decay term $(p^\prime(\bar{v})-p^\prime(v_{+}))\bar{v}_{x}$ in \cite{Nishihara-Yang1999}, Marcati, Mei and Rubino in \cite{Marcati-Mei-Rubino2005} chose a new asymptotic profile as the nonlinear diffusion waves $(\bar{v},\bar{u})(x,t)$ which satisfies
\begin{equation}\notag
 \left\{\begin{array}{l}
 \bar{v}_{t}-\bar{u}_{x}=0,\\[2mm]
  p(\bar{v})_{x}=-\alpha \bar{u},\quad (x,t) \in \mathbb{R}^{+}\times \mathbb{R}^{+},\\[2mm]
  \bar{u}|_{x=0}=0,\quad (\bar{v},\bar{u})|_{x=\infty}=(v_{+},0),
 \end{array}
   \right.	
\end{equation}
and they improved the convergence rates as $\left\|(v-\bar{v}, u-\bar{u})(t) \right\|_{L^{\infty}} \leq C(t^{-1},t^{-\frac{3}{2}})$ when the initial perturbation is small belonging to $(H^{3} \times H^{2}) \cap (L^{1} \times L^{1})$.
Such smallness assumptions on the initial perturbation were then partially relaxed by Jiang and Zhu in \cite{Jiang-Zhu2009}.
For other studies related to the initial-boundary value problem to \eqref{1.4}, we refer to \cite{Cui-Yin-Zhu-Zhu2019,Geng-Zhang2015,Lin-Mei2010} and references therein.

When $\alpha=\sigma$, $p(v)=\frac{1}{3v}, g(u)=\frac{u^{2}\sqrt{4-3u^{2}}}{2+\sqrt{4-3u^{2}}}$ and $f(v)=\frac{1}{v}$, the system \eqref{1.4} reduces to $M_{1}$ model \eqref{1.5a} which we will study in the next.
 Motivated by these preceding results, in the present paper, we will consider the convergence to nonlinear diffusion waves for solutions to the IBVP \eqref{1.1}--\eqref{1.3} on quadrant, and we will obtain a sharper result which indeed improves those in Nishihara and Yang \cite{Nishihara-Yang1999}, and Marcati, Mei and Rubino \cite{Marcati-Mei-Rubino2005}.

According to Darcy's law, we expect the asymptotic profile of the Dirichlet IBVP \eqref{1.1}--\eqref{1.3} satisfying
\begin{equation}\label{1.5}
 \left\{\begin{array}{l}
 \bar{v}_{t}-\bar{u}_{x}=0,\\[2mm]
  p(\bar{v})_{x}=-\alpha \bar{u},\quad (x,t) \in \mathbb{R}^{+}\times \mathbb{R}^{+},\\[2mm]
  \bar{u}|_{x=0}=0,\quad (\bar{v},\bar{u})|_{x=\infty}=(v_{+},0),
 \end{array}
   \right.
\end{equation}
or
\begin{equation}\label{1.6}
 \left\{\begin{array}{l}
\displaystyle \bar{v}_{t}=-\frac{1}{\alpha}p(\bar{v})_{xx},\\[2mm]
\displaystyle \bar{u}=-\frac{1}{\alpha}p(\bar{v})_{x},\quad (x,t) \in \mathbb{R}^{+}\times \mathbb{R}^{+},\\[2mm]
\bar{v}_{x}|_{x=0}=0,\quad (\bar{v},\bar{u})|_{x=\infty}=(v_{+},0).
 \end{array}
        \right.
\end{equation}
Taking the limits to system of \eqref{1.1} as $x \rightarrow \infty$, and noting that $p(v)_x$~and~ $(g(u)f(v))_{x}$ will be vanishing, then we find that $u(\infty,t)$ satisfies formally the following ODEs:
\begin{equation}\label{1.7}
 \left\{\begin{array}{l}
\displaystyle \frac{{\rm d}}{{\rm d}t}v(\infty,t)=0,\\[2mm]
\displaystyle  \frac{{\rm d}}{{\rm d}t}u(\infty,t)=-\alpha u(\infty,t),\\[2mm]
  (v,u)(\infty,0)=(v_0,u_0)(\infty)=(v_{+},u_{+}).
 \end{array}
   \right.
\end{equation}
By direct calculation, we have
\begin{equation}\label{1.8}
 \lim _{x \rightarrow \infty} (v,u)(x,t)=(v,u)(\infty, t)= (v_{+},u_{+}{\rm e}^{-\alpha t}).
\end{equation}
To eliminate the value of $u(x,t)$ at $x=\infty$, we introduce the following auxiliary functions as in \cite{Nishihara-Yang1999, Marcati-Mei-Rubino2005, Jiang-Zhu2009}
\begin{equation}\label{1.9}
 \left\{\begin{array}{l}
 \hat{v}(x,t)=\displaystyle\frac{u_{+}}{-\alpha}{\rm e}^{-\alpha t}m_{0}(x),\\[2mm]
  \hat{u}(x,t)=u_{+}{\rm e}^{-\alpha t}\displaystyle\int_{0}^{x} m_{0}(y) {\rm d}y,
 \end{array}
   \right.
\end{equation}
where $m_{0} \in C_{0}^{\infty}(\mathbb{R}^{+})$ satisfies
\begin{equation}\notag
  \int_{0}^{\infty} m_{0}(x) {\rm d}x=1.
\end{equation}
Then one can immediately obtain
\begin{equation}\label{1.10}
 \left\{\begin{array}{l}
\hat{v}_{t}-\hat{u}_{x}=0,\\[2mm]
\hat{u}_{t}=- \alpha\hat{u},\quad (x,t) \in \mathbb{R}^{+}\times \mathbb{R}^{+},\\[2mm]
\hat{u}|_{x=0}=0,\quad (\hat{v},\hat{u})|_{x=\infty}=(0,u_{+}{\rm e}^{-\alpha t}).
 \end{array}
        \right.
\end{equation}
Combining \eqref{1.1}, \eqref{1.5} and \eqref{1.10}, we obtain
\begin{equation}\label{1.11}
\left\{\begin{array}{l}
(v-\bar{v}-\hat{v})_{t}-(u-\bar{u}-\hat{u})_{x}=0,\\[2mm]
(u-\bar{u}-\hat{u})_{t}+(p(v)-p(\bar{v}))_{x}+\alpha(u-\bar{u}-\hat{u})=-\bar{u}_{t}+(g(u)f(v))_{x}.
\end{array}
   \right.
\end{equation}
Integrating $\eqref{1.11}_{1}$ over $[0,\infty) \times [0,t]$, and noticing that $\bar{v}(x,0)=v_{+}+\delta_{0}\phi_{0}(x)$ in the next section as in \cite{Marcati-Mei-Rubino2005}, we get
\begin{equation}\label{1.12}
\int_{0}^{\infty}(v-\bar{v}-\hat{v})(y, t) {\rm d}y=\int_{0}^{\infty}\left(v_{0}(x)-v_{+}\right) d x-\delta_{0} \int_{0}^{\infty} \phi_{0}(x) d x+\frac{u_{+}}{\alpha}=0.
\end{equation}
Thus, we reach the setting of perturbations
\begin{equation}\label{1.13}
\left\{\begin{array}{l}
V(x,t)= -\displaystyle\int_{x}^{\infty} \left(v-\bar{v}-\hat{v}\right)(y,t) {\rm d}y,\\[2mm]
z(x,t)=u(x,t)-\bar{u}(x,t)-\hat{u}(x,t),
\end{array}
   \right.
\end{equation}
then after the integration of $\eqref{1.11}_{1}$ over $(x,\infty)$, we have the reformulated problem
\begin{equation}\label{1.14}
 \left\{\begin{array}{l}
V_{t}-z=0,\\[2mm]
z_{t}+ \left(p(V_{x}+\bar{v}+\hat{v})-p(\bar{v})\right)_{x}+\alpha z=(g(z+\bar{u}+\hat{u})f(V_{x}+\bar{v}+\hat{v}))_{x},
 \end{array}
        \right.
 \end{equation}
 and the linearized problem around $\bar{v}$ gives
 \begin{equation}\label{1.15}
 \left\{\begin{array}{l}
V_{t}-z=0,\\[2mm]
z_{t}+\left(p^\prime(\bar{v})V_{x}\right)_{x}+\alpha z=F_{1}+F_{2},\\[2mm]
(V,z)|_{t=0} :=(V_0,z_0)(x),\\[2mm]
V|_{x=0}=0,
 \end{array}
        \right.
 \end{equation}
 or
\begin{equation}\label{1.16}
 \left\{\begin{array}{l}
V_{tt}+\left(p^\prime(\bar{v})V_{x}\right)_{x}+\alpha V_{t}=F_{1}+F_{2},\\[2mm]
(V,V_{t})|_{t=0} :=(V_0,z_0)(x),\\[2mm]
V|_{x=0}=0,
\end{array}
        \right.
\end{equation}
 where
 \begin{equation}\label{1.17}
{F}_{1}=\frac{1}{\alpha}p(\bar{v})_{xt}-\left(p(V_{x}+\bar{v}+\hat{v})-p(\bar{v})-p^\prime(\bar{v})V_{x} \right)_{x},
\end{equation}
\begin{equation}\label{1.18}
{F}_{2}= \left(g(V_{t}+\bar{u}+\hat{u})f(V_{x}+\bar{v}+\hat{v})\right)_{x},
\end{equation}
\begin{equation}\label{1.19}
\left(V_{0}, z_{0}\right)(x)=\left(-\int_{x}^{\infty}\left(v_{0}(y)-\bar{v}(y, 0)-\hat{v}(y, 0)\right) {\rm d}y, ~~ u_{0}(x)-\bar{u}(x, 0)-\hat{u}(x, 0)\right).
\end{equation}

\vspace{4mm}

{\bf Notations.} In the following, $C$ and $c$ denote the generic positive constants depending only on the initial data,  but independent of the time. For function spaces, $L^{p}=L^{p}(\mathbb{R}^{+}) ~ (1\leq p\leq \infty)$ is an usual Lebesgue space with the norm
\begin{equation}\notag
 \|f\|_{L^{p}}=\left(\int_{\mathbb{R}^{+}}|f(x)|^{p}{\rm d}x\right)^{\frac{1}{p}}, ~~~ 1\leq p< \infty~~~\mbox{and} ~~~ \|f\|_{L^{\infty}}=\sup \limits_{\mathbb{R}^{+}}|f(x)|.
\end{equation}
For any integer $l \geq 0$, $H^{l}$ denotes the usual $l$-th order Sobolev space on $\mathbb{R}^{+}$ with its norm
\begin{equation}\notag
 \|f\|_{l}=\left(\sum_{j=0}^{l}\left\|\partial_{x}^{j} f\right\|^{2}\right)^{\frac{1}{2}}, ~~~ \quad\|\cdot\|=\|\cdot\|_{0}=\|\cdot\|_{L^{2}}.
\end{equation}
In order to state our main results, we assume that the following assumptions hold:
\begin{align}
&\label{1.21}p,f \in C^{3}(\mathbb{R}^{+}), ~~ p^\prime<0~ \mbox{for}~ \mbox{any}~ v>0, \\
&\label{1.22}g \in C^{3}(\mathbb{R}), ~~~ g(0)=g^\prime (0)=0.
\end{align}

The followings  are the  main results.
\begin{theorem}\label{Thm1}
 Suppose that \eqref{1.21}--\eqref{1.22} hold, $\|V_{0}\|_{H^{3}}+\|z_{0}\|_{H^{2}}$ and $\delta:=\|v_{0}-v_{+}\|_{L^{1}}+|u_{+}|$ are sufficiency small. Then, there exists a unique time-global solution $V(x,t)$ of \eqref{1.15}, which satisfies
 \begin{equation}\notag
   V(x,t) \in C^{k}(0,\infty; H^{3-k}), ~~~ k=0,1,2,3; \quad z(x,t) \in C^{k}(0,\infty; H^{2-k}), ~~~ k=0,1,2,
  \end{equation}
  and
 \begin{align}
  \label{1.23}&\|\partial_{x}^{k}V(t)\| \leq C(1+t)^{-\frac{k}{2}},\qquad~~~~~~~~ 1\leq k\leq 3,\\
  \label{1.24}&\|\partial_{x}^{k}\partial_{t}^{j}z(t)\|\leq C(1+t)^{-\frac{k}{2}-j-1},\qquad  0\leq k+j\leq 2,~0\leq j\leq 1,\\
  \label{1.25}&\|\partial_{t}^{2}z(t)\|\leq C(1+t)^{-\frac{5}{2}}.
  \end{align}
\end{theorem}
\begin{theorem}\label{Thm2}
 Under the assumptions of Theorem \ref{Thm1}, we assume further that $(V_{0}+\frac{1}{\alpha}z_{0})(x) \in L^{1}$, then solution $V(x,t)$ obtained in Theorem \ref{Thm1} satisfies the following improved decay rates
 \begin{align}
  \label{1.26}&\|\partial_{x}^{k}V(t)\| \leq C(1+t)^{-\frac{k}{2}-\frac{1}{4}},\qquad~~~~~ 0\leq k\leq 3,\\
  \label{1.27}&\|\partial_{x}^{k}\partial_{t}^{j}z(t)\|\leq C(1+t)^{-\frac{k}{2}-j-\frac{5}{4}},\qquad  0\leq k+j\leq 2,~0\leq j\leq 1,\\
  \label{1.28}&\|\partial_{t}^{2}z(t)\|\leq C(1+t)^{-\frac{11}{4}}.
  \end{align}
Moreover, if $u_{+}=0$, $\int_{0}^{\infty}(V_{0}+\frac{1}{\alpha}z_{0})(x){\rm d}x=0$ and $W_{0}(x):=\int_{0}^{x}(V_{0}+\frac{1}{\alpha}z_{0})(y){\rm d}y \in L^{1}$, then faster decay rates hold:
\begin{align}
  \label{1.29}&\|\partial_{x}^{k}V(t)\| \leq C(1+t)^{-\frac{k}{2}-\frac{3}{4}},\qquad~~~~~ 0\leq k\leq 3,\\
  \label{1.30}&\|\partial_{x}^{k}\partial_{t}^{j}z(t)\|\leq C(1+t)^{-\frac{k}{2}-j-\frac{7}{4}},\qquad  0\leq k+j\leq 2,~0\leq j\leq 1,\\
  \label{1.31}&\|\partial_{t}^{2}z(t)\|\leq C(1+t)^{-\frac{13}{4}}.
  \end{align}
\end{theorem}
\begin{remark}
As we can see from Theorem 2.1 in Marcati, Mei and Rubino \cite{Marcati-Mei-Rubino2005}, the authors required that $\|V_{0}\|_{H^{3}}+\|z_{0}\|_{H^{2}}+\|v_{0}-v_{+}\|_{L^{1}}+|u_{+}|+\|V_{0}\|_{L^{1}}+\|z_{0}\|_{L^{1}}$ be small enough to get the decay estimates $\|\partial_{x}^{k}\partial_{t}^{l}V(t)\| \leq C(1+t)^{-\frac{1}{4}-\frac{k}{2}-l}$ in the case of $0\leq k \leq 2$ or $0\leq l \leq 1$, but estimates for other cases have not been elucidated. In contrast, our conditions are weaker, and the conclusions are sharper.	
\end{remark}
\begin{remark}
In fact, except for $\|z_{tt}(t)\|$, the decay estimates obtained in Theorem \ref{Thm2} are all optimal. As for $\|z_{tt}(t)\|$, we can also use the similar way to obtain an extra time-decay $(1+t)^{-\frac{1}{2}}$ when $\left(V_{0},z_{0}\right)(x)$ is small belonging to $H^{4} \times H^{3}$. 	
\end{remark}
\begin{remark}
 As for compressible Euler equations with damping, the condition $u_{+}=0$ in Theorem \ref{Thm2} can be removed. One can see \eqref{3.140} in Section \ref{s3} for more details.	
\end{remark}

Before concluding this section, we point out the main difference between the study in this paper and the related results in the previous works. Similar arguments used in \cite{Nishihara-Yang1999, Marcati-Mei-Rubino2005}, we will prove Theorem \ref{Thm1} by deriving the key {\it a priori} energy estimates.
But for the way to obtain the decay estimates \eqref{1.26}--\eqref{1.28}, our analyses is quite different from \cite{Marcati-Mei-Rubino2005}.
From the dissussion in \cite{Marcati-Mei-Rubino2005}, we can see that once they got the existence and decay rates of the solution in the $L^{2}$-framework, they used Green's function to get an integral representation of the solution, and only analyzed the integral expression of the solution to obtain the improved decay rates $\|\partial_{x}^{k}\partial_{t}^{l}V(t)\| \leq C(1+t)^{-\frac{1}{4}-\frac{k}{2}-l}$ in the case of $0\leq k \leq 2$ or $0\leq l \leq 1$, but other cases cann't be clarified.
Unlike their approach, by analyzing the integral representation of the solution, we firstly obtain $\|V(t)\| \leq C(1+t)^{-\frac{1}{4}}$. With all these preparations, by using the time-weighted energy estimates, we can get the desired estimates \eqref{1.26}--\eqref{1.28}. Similar analysis ideas can also deduce \eqref{1.29}--\eqref{1.31}. See Section \ref{s3} for more detials.
This technique is quite useful, it has been successfully used in \cite{Geng-Zhang2015}. We think this approach has at least two advantages: one is that we can get the desired decay rates without putting additional regularity on the initial data, and the other is that the calculation process is much simpler and clearer.
Finally, we want to point out that for the Newman boundary problem, due to the difficulties in dealing with some boundary terms, we have not found a suitable method, but some ideas in this paper provide some key clues, and we will continue to study this problem in the future.

The arrangement of the present paper is as follows. In Section \ref{s2}, we prepare some preliminaries, which are useful in the proof of theorems. Section \ref{s3} is devoted to the proof of our main theorems.

\section{Preliminaries}\label{s2}
In this section, we state some known results which will be used for the proof of our main results in next section.

Firstly, as for the solution $(\bar{v},\bar{u})(x,t)$ of \eqref{1.6}, we can obtain that
\begin{lemma}\label{l2.1}
{\rm (see \cite{Jiang-Zhu2009, Marcati-Mei-Rubino2005})} Let $\bar{v}(x,t)$ be a solution of \eqref{1.6} with the initial data $\bar{v}(x,0)=v_{+}+\delta_{0}\phi_{0}$, where $\delta_{0}$ is a constant satisfying
\begin{equation}\label{2.1}
 \int_{0}^{\infty}\left[v_{0}(x)-v_{+}\right] d x-\delta_{0} \int_{0}^{\infty} \phi_{0}(x) {\rm d} x+\frac{u_{+}}{\alpha}=0,
\end{equation}
and $\phi_{0}$ is a given smooth function satisfying
\begin{equation}\notag
 \phi_{0}\in L^{1}(\mathbb{R}^{+}) ~~~~ \mbox{and} ~~~~ \int_{0}^{\infty}\phi_{0} {\rm d}x\neq0.
\end{equation}
Then the solution $(\bar{v},\bar{u})(x,t)$ of \eqref{1.6} globally and uniquely exists and satisfies
\begin{align}\label{2.2}
&\|\partial^{k}_{x}\partial^{j}_{t}(\bar{v}-v_{+})(t)\|\leq C|\delta_{0}|(1+t)^{-\frac{4j+2k+1}{4}},\nonumber\\
&\|\partial^{k}_{x}\partial^{j}_{t}(\bar{v}-v_{+})(t)\|_{L^{1}}\leq C|\delta_{0}|(1+t)^{-\frac{2j+k}{2}},\nonumber\\
&\|\partial^{k}_{x}\partial^{j}_{t}(\bar{v}-v_{+})(t)\|_{L^{\infty}}\leq C|\delta_{0}|(1+t)^{-\frac{2j+k+1}{2}},\qquad k,j\geq 0.
  \end{align}
\end{lemma}
\begin{remark}\label{r2.1}
It should be noted that from \eqref{2.1}, we can immediately obtain
 \begin{equation}\notag
|\delta_{0}|\leq C(\|v_{0}-v_{+}\|_{L^{1}}+|u_{+}|).
\end{equation}
\end{remark}
Next, from \eqref{1.9}, one can confirmed that the correction function $(\hat{v},\hat{u})(x,t)$ satisfies
\begin{lemma}\label{l2.2}
 Let $k, j$ be nonnegative integers and $p \in [1, \infty]$ is an integer, it holds that
 \begin{equation}\label{2.3}
  \begin{split}
    &\left\|\partial_{x}^{k} \partial_{t}^{j} \hat{v}(t)\right\|_{L^{p}}\leq C|u_{+}| {\rm e}^{-\alpha t}, \quad k \geq 0, j \geq 0,\\
    &\left\|\partial_{x}^{k} \partial_{t}^{j} \hat{u}(t)\right\|_{L^{p}}\leq C|u_{+}|{\rm e}^{-\alpha t}, \quad k \geq 1, j \geq 0,\\
    & \left\|\hat{u}(t)\right\|_{L^{\infty}}\leq |u_{+}|{\rm e}^{-\alpha t}.
    \end{split}
  \end{equation}
\end{lemma}

Finally, we introduce the Sobolev inequation.
\begin{lemma}\label{l2.3}
 Let $f \in H^{1}$, then
 \begin{equation}\label{2.4}
  \|f\|_{L^{\infty}} \leq \sqrt{2}\|f\|^{\frac{1}{2}}\|f_{x}\|^{\frac{1}{2}} .
  \end{equation}
\end{lemma}

\section{Proof of the main theorems}\label{s3}

In this section, we will devote ourselves to proof our main results: Theorem \ref{Thm1}--Theorem \ref{Thm2}. In the first subsection, we shall prove the global existence, uniqueness and time decay rates by deriving {\it a priori} estimates in the $L^{2}$-framework, and Theorem \ref{Thm1} is obtained. In the second subsection, we shall use the Green's method combined with the technical time-weighted energy estimates to obtain the improved decay rates \eqref{1.26}--\eqref{1.31} in Theorem \ref{Thm2}. In what follows, we can put $\alpha=1$ without loss of generality, and denote $g(u)$ and $f(v)$ by $g$ and $f$ respectively.

\subsection{Proof of Theorem \ref{Thm1}}\label{s3.1}

The main purpose of this subsection is to study global existence and uniqueness of solutions to \eqref{1.15} in the $L^{2}$-framework, and obtain \eqref{1.23}--\eqref{1.25}. It is well known that the global existence can be obtained by the continuation argument based on the local existence of solutions and {\it a priori} estimates. As for \eqref{1.15}, the local existence can be proved by the standard iteration method (cf. \cite{Matsumura1977}) and the details will be omitted here for brevity. In the next, we are devoted to establishing the following {\it a priori} estimates.
\begin{proposition}\label{p1}
Under the assumptions in Theorem \ref{Thm1}, there are positive constants $\varepsilon$ and $C$ such that if the smooth solution $(V,z)(x,t)$ to \eqref{1.15} on $0 \leq t \leq T$ for $T>0$ satisfies
\begin{equation}\label{3.1}
\begin{split}
 N(T):=\sup \limits_{0 \leq t \leq T}&\left\{\sum \limits_{k=0}^{3}(1+t)^{k}\|\partial_{x}^{k}V(t)\|^{2}+ \sum \limits_{k=0}^{2}(1+t)^{k+2}\|\partial_{x}^{k}z(t)\|^{2} \right. \\
 &~~~~~¡¤\left.+\sum \limits_{k=0}^{1}(1+t)^{k+4}\|\partial_{x}^{k}z_{t}(t)\|^{2}\right \} \leq \varepsilon^{2},
 \end{split}
\end{equation}
 then it holds that
\begin{align}\label{3.2}
 \sum_{k=0}^{3}&(1+t)^{k}\left\|\partial_{x}^{k} V(t)\right\|^{2}+\sum_{k=0}^{2}(1+t)^{k+2}\left\|\partial_{x}^{k} z(t)\right\|^{2} \nonumber\\
&+\int_{0}^{t}\left[\sum_{j=1}^{3}(1+s)^{j-1}\left\|\partial_{x}^{j} V(s)\right\|^{2}+\sum_{j=0}^{2}(1+s)^{j+1}\left\|\partial_{x}^{j} z(s)\right\|^{2}\right]{\rm d}s \nonumber\\
\leq&C\left(\left\|V_{0}\right\|_{3}^{2}+\left\|z_{0}\right\|_{2}^{2}+\delta\right),
 \end{align}
 and
 \begin{align}\label{3.3}
 (1+t)^{4}&\left\|z_{t}(t)\right\|^{2}+(1+t)^{5}(\left\|z_{xt}(t)\right\|^{2}+\left\|z_{tt}(t)\right\|^{2}) \nonumber\\
&+\int_{0}^{t}\left[(1+s)^{4}\left\|z_{xt}(s)\right\|^{2}+(1+s)^{5}\left\|z_{tt}(s)\right\|^{2}\right){\rm d}s \nonumber\\
\leq&C\left(\left\|V_{0}\right\|_{3}^{2}+\left\|z_{0}\right\|_{2}^{2}+\delta\right).
 \end{align}
\end{proposition}
By applying \eqref{2.4} and  \eqref{3.1}, we have
\begin{equation}\label{3.4}
\begin{split}
 &\|\partial_{x}^{k}V(t)\|_{L^{\infty}} \leq \sqrt{2}\varepsilon (1+t)^{-\frac{1}{4}-\frac{k}{2}},\quad k=0,1,2,\\
 &\|\partial_{x}^{k}z(t)\|_{L^{\infty}} \leq \sqrt{2}\varepsilon (1+t)^{-\frac{5}{4}-\frac{k}{2}},\quad k=0,1,\\
 &\|z_{t}(t)\|_{L^{\infty}} \leq \sqrt{2}\varepsilon (1+t)^{-\frac{9}{4}},
 \end{split}
\end{equation}
which will be used later. From \eqref{1.16} and \eqref{1.22}, one can immediately obtain the following boundary conditions
\begin{equation}\label{3.5}
V(0,t)=V_{t}(0,t)=V_{xx}(0,t)=V_{xxt}(0,t)=0,~~~\mbox{etc}.
\end{equation}

Then we introduce the following lemma, which will play a key role in the proof of Proposition \ref{p1}.
\begin{lemma}\label{l3.1}
Assume that all the conditions in Proposition \ref{p1} hold, then it holds that
\begin{equation}\label{3.6a}
\left\{\begin{array}{l}
 |(g^\prime f)(x,t)|\leq C(\varepsilon+\delta)(1+t)^{-1},\\[2mm]
|(g^\prime f)_{x}(x,t)|\leq C(\varepsilon+\delta)(1+t)^{-\frac{3}{2}},~~~|(g^\prime f)_{t}(x,t)|\leq C(\varepsilon+\delta)(1+t)^{-2},\\[2mm]
|(g^\prime f)_{xx}(x,t)|\leq C|V_{xxt}(x,t)|+C(1+t)^{-1}|V_{xxx}(x,t)|+C(\varepsilon+\delta)(1+t)^{-2},\\[2mm]
|(g^\prime f)_{xt}(x,t)|\leq C|V_{xtt}(x,t)|+C(1+t)^{-1}|V_{xxt}(x,t)|+C(\varepsilon+\delta)(1+t)^{-\frac{5}{2}},\\[2mm]
|(g^\prime f)_{tt}(x,t)|\leq C|V_{ttt}(x,t)|+C(1+t)^{-1}|V_{xtt}(x,t)|+C(\varepsilon+\delta)(1+t)^{-3},
\end{array}
        \right.	
\end{equation}
and
\begin{equation}\label{3.6b}
\left\{\begin{array}{l}
 |(g f^\prime)(x,t)|\leq C(\varepsilon+\delta)(1+t)^{-2},\\[2mm]
|(g f^\prime)_{x}(x,t)|\leq C(\varepsilon+\delta)(1+t)^{-\frac{5}{2}},~~~|(g f^\prime)_{t}(x,t)|\leq C(\varepsilon+\delta)(1+t)^{-3},\\[2mm]
|(g f^\prime)_{xx}(x,t)|\leq C(1+t)^{-1}|V_{xxt}(x,t)|+C(1+t)^{-2}|V_{xxx}(x,t)|+C(\varepsilon+\delta)(1+t)^{-3},\\[2mm]
|(g f^\prime)_{xt}(x,t)|\leq C(1+t)^{-1}|V_{xtt}(x,t)|+C(1+t)^{-2}|V_{xxt}(x,t)|+C(\varepsilon+\delta)(1+t)^{-\frac{7}{2}},\\[2mm]
|(g f^\prime)_{tt}(x,t)|\leq C(1+t)^{-1}|V_{ttt}(x,t)|+C(1+t)^{-2}|V_{xtt}(x,t)|+C(\varepsilon+\delta)(1+t)^{-4}.
\end{array}
        \right.	
\end{equation}
\end{lemma}
\begin{proof}
By direct calculation, we can easily get from \eqref{1.21}--\eqref{1.22} and Taylor's expansion that
\begin{equation}\notag
  \begin{split}
    &\left|g^\prime f\right|\leq C|u|, ~~~~~ |(g f^\prime)|\leq C|u^{2}|, \\
    &\left|(g^\prime f)_{i}\right|\leq C(|u_{i}|+|uv_{i}|),~~~|(g f^\prime)_{i}|\leq C(|uu_{i}|+|u^{2}v_{i}|),\\
    &\left|\left(g^\prime f\right)_{ij}\right| \leq C(|u_{i}u_{j}|+|u_{ij}|+|u_{i}v_{j}|+|u_{j}v_{i}|+|uv_{i}v_{j}|+|uv_{ij}|),\\
    &\left|\left(g f^\prime\right)_{ij}\right| \leq C(|u_{i}u_{j}|+|uu_{ij}|+|uu_{i}v_{j}|+|uu_{j}v_{i}|+|u^{2}v_{i}v_{j}|+|u^{2}v_{ij}|),
    \end{split}
  \end{equation}
for $i,j=x ~\mbox{or}~ t$. Noticing that $u=V_{t}+\bar{u}+\hat{u}$ and $v=V_{x}+\bar{v}+\hat{v}$, by using $\eqref{1.6}_{2}$, \eqref{2.2}--\eqref{2.4} and \eqref{3.4}, one can immediately obtain \eqref{3.6a}--\eqref{3.6b}.
\end{proof}

With the obove preparations in hand, we now turn to prove Proposition \ref{p1}, which will be obtained through a series of lemmas.
\begin{lemma}\label{l3.2}
If $N(T) \leq \varepsilon^{2}$ and $\delta$ are sufficiently small, then it holds that
\begin{equation}\label{3.7}
\begin{split}
   \|V(t)\|^{2}+&(1+t)(\|V_{x}(t)\|^{2}+\|V_{t}(t)\|^{2})+\int_{0}^{t}(1+s)(\|V_{x}(s)\|^{2}+\|V_{t}(s)\|^{2}){\rm d}s\\
   & \leq  C\left(\|V_{0}\|^{2}_{1}+\|z_{0}\|_{1}^{2}+ \delta \right),
   \end{split}
  \end{equation}
  for $0 \leq t \leq T$.
\end{lemma}
\begin{proof}
Firstly, multiplying $\eqref{1.16}_{1}$ by $V$ and integrating it with respect to $x$ over $\mathbb{R}^{+}$, we can get
 \begin{align}\label{3.8}
   \frac{{\rm d}}{{\rm d}t}\int_{0}^{\infty} \left(\frac{V^{2}}{2}+VV_{t}\right){\rm d} x-\int_{0}^{\infty} p^\prime(\bar{v})V_{x}^{2}{\rm d}x =\int_{0}^{\infty}V_{t}^{2}{\rm d}x+\int_{0}^{\infty}F_{1}V{\rm d}x+\int_{0}^{\infty}F_{2}V{\rm d} x .
 \end{align}
 The right hand side of \eqref{3.8} can be estimated as follows. From \eqref{2.2}--\eqref{2.4} and \eqref{3.4}--\eqref{3.5}, we have
 \begin{equation}\label{3.9}
  \begin{split}
   \int_{0}^{\infty} F_{1}V{\rm d} x &= \int_{0}^{\infty}\left[-p(\bar{v})_{t}+p(V_{x}+\bar{v}+\hat{v})-p(\bar{v})-p^\prime(\bar{v})V_{x}\right]V_{x}{\rm d} x \\&
   \leq -\frac{p^\prime(\bar{v})}{16} \|V_{x}(t)\|^{2}+C\delta(1+t)^{-\frac{5}{2}}.
   \end{split}
  \end{equation}
  Noticing that
\begin{equation}\label{3.10}
   F_{2}= g^\prime f(V_{xt}-p(\bar{v})_{xx}+\hat{v}_{t})+gf^\prime(V_{xx}+\bar{v}_{x}+\hat{v}_{x}),
  \end{equation}
  then
\begin{align}\label{3.11}
   \int_{0}^{\infty}F_{2}V{\rm d}x=&\int_{0}^{\infty}[g^\prime f(V_{xt}-p(\bar{v})_{xx}+\hat{v}_{t})+gf^\prime(V_{xx}+\bar{v}_{x}+\hat{v}_{x})]V{\rm d}x\nonumber\\
   =&\int_{0}^{\infty}g^\prime fV_{xt}V{\rm d}x+\int_{0}^{\infty}g^\prime f(-p(\bar{v})_{xx}+\hat{v}_{t})V{\rm d}x+\int_{0}^{\infty}gf^\prime(V_{xx}+\bar{v}_{x}+\hat{v}_{x})V{\rm d}x\nonumber\\
   :=&I_{1}+I_{2}+I_{3}.
  \end{align}
  From \eqref{1.21}--\eqref{1.22}, \eqref{2.2}--\eqref{2.4}, \eqref{3.4}--\eqref{3.5} and Taylor's expansion, we can deduce that
  \begin{align}\label{3.12}
   I_{1}=&\int_{0}^{\infty}g^\prime fV_{xt}V{\rm d}x\nonumber\\
   =&-\int_{0}^{\infty}g^\prime fV_{t}V_{x}{\rm d}x-\int_{0}^{\infty}VV_{t}[g^{\prime\prime} f(V_{xt}-p(\bar{v})_{xx}+\hat{v}_{t})+g^\prime f^\prime(V_{xx}+\bar{v}_{x}+\hat{v}_{x})]{\rm d}x\nonumber\\
   \leq&C\int_{0}^{\infty}|V_{t}V_{x}|{\rm d}x+\frac{1}{2}\int_{0}^{\infty}V_{t}^{2}(Vg^{\prime\prime}f)_{x}{\rm d}x+C\int_{0}^{\infty}|VV_{t}|(\bar{v}_{x}^{2}+|\bar{v}_{xx}|+|\hat{v}_{t}|){\rm d}x \nonumber\\
   &+C\int_{0}^{\infty}|VV_{t}|(|V_{t}|+|\bar{v}_{x}|+|\hat{u}|)(|V_{xx}|+|\bar{v}_{x}|+|\hat{v}_{x}|){\rm d}x \nonumber\\
   \leq& -\frac{p^\prime(\bar{v})}{32}\int_{0}^{\infty}V_{x}^{2}{\rm d}x+C\int_{0}^{\infty}V_{t}^{2}{\rm d}x+C\|V(t)\|_{L^{\infty}}^{2}\int_{0}^{\infty}(\bar{v}_{x}^{4}+\bar{v}_{xx}^{2}+\hat{v}_{t}^{2}){\rm d}x\nonumber\\
   &+C\|VV_{xx}(t)\|_{L^{\infty}}^{2}\int_{0}^{\infty}\bar{v}_{x}^{2}{\rm d}x+C\int_{0}^{\infty}\hat{v}_{x}^{2}{\rm d}x+\|\hat{u}(t)\|_{L^{\infty}}^{2}\int_{0}^{\infty}(V_{xx}^{2}+\bar{v}_{x}^{2}){\rm d}x\nonumber\\
   \leq& -\frac{p^\prime(\bar{v})}{16}\|V_{x}(t)\|^{2} +C\|V_{t}(t)\|^{2}+ C\delta (1+t)^{-3},
  \end{align}
  \begin{align}\label{3.13}
   I_{2}=&\int_{0}^{\infty}g^\prime f(-p(\bar{v})_{xx}+\hat{v}_{t})V{\rm d}x\nonumber\\
   \leq&C\int_{0}^{\infty}(|V_{t}|+|\bar{v}_{x}|+|\hat{u}|)(\bar{v}_{x}^{2}+|\bar{v}_{xx}|+|\hat{v}_{t}|)|V|{\rm d}x\nonumber\\
   \leq&\|V_{t}(t)\|^{2}+C\delta (1+t)^{-3}\|V(t)\|^{2}+C\delta (1+t)^{-\frac{3}{2}}\|\bar{v}_{x}(t)\|\|V(t)\|\nonumber\\
   &+C\|\hat{u}(t)\|_{L^{\infty}}(\|\bar{v}_{x}(t)\|^{2}+\|\bar{v}_{xx}(t)\|_{L^{1}}+\|\hat{v}_{t}(t)\|_{L^{1}})\nonumber\\
   \leq&\|V_{t}(t)\|^{2}+C\delta (1+t)^{-\frac{9}{4}},
  \end{align}
  and
\begin{align}\label{3.14}
    &\int_{0}^{\infty}gf^\prime(V_{xx}+\bar{v}_{x}+\hat{v}_{x})V{\rm d}x\nonumber\\
    \leq&C\int_{0}^{\infty}(V_{t}-p(\bar{v})_{x}+\hat{u})^{2}(|V_{xx}|+|\bar{v}_{x}|+|\hat{v}_{x}|)|V|{\rm d}x\nonumber\\
    \leq&C\|V_{t}(t)\|^{2}+C\int_{0}^{\infty}[|V_{t}|(|\bar{v}_{x}|+|\hat{u}|)+|\bar{v}_{x}|^{2}+|\bar{v}_{x}||\hat{u}|](|V_{xx}|+|\bar{v}_{x}|+|\hat{v}_{x}|)|V|{\rm d} x \nonumber\\
  &+C\|\hat{u}(t)\|_{L^{\infty}}^{2}\int_{0}^{\infty}|V_{xx}V|{\rm d} x+C\|\hat{u}(t)\|_{L^{\infty}}^{2}\int_{0}^{\infty}|V|(|\bar{v}_{x}|+|\hat{v}_{x}|){\rm d} x\nonumber\\
  \leq& C\|V_{t}(t)\|^{2}+C\delta (1+t)^{-4}+C\delta (1+t)^{-2}(\|\bar{v}_{x}(t)\|+\|V_{xx}(t)\|)\|V(t)\|\nonumber\\
  \leq& C\|V_{t}(t)\|^{2}+ C\delta (1+t)^{-\frac{11}{4}}.
  \end{align}
  Inserting \eqref{3.9} and \eqref{3.11}--\eqref{3.14} into \eqref{3.8}, we can conclude that
\begin{align}\label{3.15}
   \frac{1}{2}\frac{{\rm d}}{{\rm d}t}\int_{0}^{\infty} \left(V^{2}+2VV_{t}\right){\rm d} x-\frac{3}{4}\int_{0}^{\infty} p^\prime(\bar{v})V_{x}^{2}{\rm d}x \leq C\|V_{t}(t)\|^{2}+ C\delta (1+t)^{-\frac{9}{4}}.
 \end{align}
Next, multiplying $\eqref{1.16}_{1}$ by $V_{t}$ and integrating it with respect to $x$ over $\mathbb{R}^{+}$, we have
\begin{equation}\label{3.16}
  \begin{split}
  \frac{{\rm d}}{{\rm d}t}&\int_{0}^{\infty} \left(\frac{V_{t}^{2}}{2}-\frac{p^\prime(\bar{v})}{2}V_{x}^{2}\right){\rm d}x+\int_{0}^{\infty}V_{t}^{2}{\rm d}x= -\int_{0}^{\infty}\frac{p^{\prime\prime}(\bar{v})\bar{v}_{t}V_{x}^{2}}{2}{\rm d} x+\int_{0}^{\infty}F_{1}V_{t}{\rm d}x+\int_{0}^{\infty}F_{2}V_{t}{\rm d}x.
  \end{split}
  \end{equation}
The right hand side of \eqref{3.16} can be estimated as follows. From \eqref{2.2}--\eqref{2.4} and \eqref{3.4}--\eqref{3.5}, we get
\begin{equation}\label{3.17}
   -\int_{0}^{\infty} \frac{p^{\prime\prime}(\bar{v})\bar{v}_{t}V_{x}^{2}}{2}{\rm d}x \leq C\delta (1+t)^{-\frac{3}{2}}\|V_{x}(t)\|^{2},	
   \end{equation}
   and
\begin{align}\label{3.18}
    \int_{0}^{\infty} F_{1}V_{t}{\rm d}x=&\int_{0}^{\infty} p(\bar{v})_{xt}V_{t}{\rm d} x+ \frac{{\rm d}}{{\rm d}t}\int_{0}^{\infty}\left[\int_{\bar{v}}^{V_{x}+\bar{v}+\hat{v}}p(s){\rm d}s-p(\bar{v})V_{x}-\frac{p^\prime(\bar{v})V_{x}^{2}}{2}\right]{\rm d}x \nonumber\\&
    +\int_{0}^{\infty}\left(-p(V_{x}+\bar{v}+\hat{v})+p(\bar{v})+p^\prime(\bar{v})V_{x}+\frac{p^{\prime\prime}(\bar{v})}{2}V_{x}^{2}\right)\bar{v}_{t}{\rm d} x\nonumber\\
    &-\int_{0}^{\infty}p(V_{x}+\bar{v}+\hat{v})\hat{v}_{t}{\rm d} x\nonumber\\
    \leq &\frac{{\rm d}}{{\rm d}t}\int_{0}^{\infty}\left[\int_{\bar{v}}^{V_{x}+\bar{v}+\hat{v}}p(s){\rm d}s-p(\bar{v})V_{x}-\frac{p^\prime(\bar{v})V_{x}^{2}}{2}\right]{\rm d}x\nonumber\\&
    +\frac{1}{16}\|V_{t}(t)\|^{2}+C(\varepsilon +\delta) (1+t)^{-\frac{3}{2}} \|V_{x}(t)\|^{2}+C\delta (1+t)^{-\frac{7}{2}}.
  \end{align}
Now we turn to deal with the last term of the righthand side of \eqref{3.18}. Notice that
  \begin{align}\label{3.19}
   \int_{0}^{\infty}F_{2}V_{t}{\rm d}x=&\int_{0}^{\infty}g^\prime f(V_{xt}-p(\bar{v})_{xx}+\hat{v}_{t})V_{t}{\rm d}x+\int_{0}^{\infty}gf^\prime(V_{xx}+\bar{v}_{x}+\hat{v}_{x})V_{t}{\rm d}x\nonumber\\
   :=&I_{4}+I_{5}.
  \end{align}
  From \eqref{1.21}--\eqref{1.22}, \eqref{2.2}--\eqref{2.4}, \eqref{3.4}--\eqref{3.5} and Taylor's expansion, we derive
  \begin{align}\label{3.20}
   I_{4}=&\int_{0}^{\infty}g^\prime f(V_{xt}-p(\bar{v})_{xx}+\hat{v}_{t})V_{t}{\rm d}x\nonumber\\
   \leq&\int_{0}^{\infty}g^\prime fV_{xt}V_{t}{\rm d}x+C\int_{0}^{\infty}(|V_{t}|+|\bar{v}_{x}|+|\hat{u}|)(\bar{v}_{x}^{2}+|\bar{v}_{xx}|+|\hat{v}_{t}|)|V_{t}|{\rm d}x\nonumber\\
   \leq&-\frac{1}{2}\int_{0}^{\infty}(g^\prime f)_{x}V_{t}^{2}{\rm d}x+\frac{1}{32}\int_{0}^{\infty}V_{t}^{2}{\rm d}x+C(1+t)^{-2}\int_{0}^{\infty}(\bar{v}_{x}^{4}+\bar{v}_{xx}^{2}+\hat{v}_{t}^{2}){\rm d}x\nonumber\\
   \leq&\frac{1}{16}\|V_{t}(t)\|^{2}+C\delta(1+t)^{-\frac{9}{2}},
  \end{align}
  and
  \begin{align}\label{3.21}
    I_{5}=&\int_{0}^{\infty}gf^\prime(V_{xx}+\bar{v}_{x}+\hat{v}_{x})V_{t}{\rm d}x\nonumber\\
    \leq&C\int_{0}^{\infty}(V_{t}-p(\bar{v})_{x}+\hat{u})^{2}(|V_{xx}|+|\bar{v}_{x}|+|\hat{v}_{x}|)|V_{t}|{\rm d}x\nonumber\\
    \leq&C(\varepsilon +\delta)\int_{0}^{\infty}V_{t}^{2}{\rm d}x+C\int_{0}^{\infty}(|\bar{v}_{x}|^{2}+|\bar{v}_{x}||\hat{u}|)(|V_{xx}|+|\bar{v}_{x}|+|\hat{v}_{x}|)|V_{t}|{\rm d} x \nonumber\\
  &+C\|\hat{u}(t)\|_{L^{\infty}}^{2}\int_{0}^{\infty}|V_{xx}V_{t}|{\rm d} x+C\|\hat{u}(t)\|_{L^{\infty}}^{2}\int_{0}^{\infty}|V_{t}|(|\bar{v}_{x}|+|\hat{v}_{x}|){\rm d} x\nonumber\\
  \leq&\frac{1}{16}\|V_{t}(t)\|^{2}+C\delta (1+t)^{-\frac{11}{2}}.
  \end{align}
  Substituting \eqref{3.17}--\eqref{3.21} into \eqref{3.16} deduce
\begin{equation}\label{3.22}
  \begin{split}
  &\frac{1}{2}\frac{{\rm d}}{{\rm d}t}\int_{0}^{\infty} \left(V_{t}^{2}-p^\prime(\bar{v})V_{x}^{2}\right){\rm d}x+\frac{3}{4}\int_{0}^{\infty}V_{t}^{2}{\rm d}x\\
  \leq&\frac{{\rm d}}{{\rm d}t}\int_{0}^{\infty}\left[\int_{\bar{v}}^{V_{x}+\bar{v}+\hat{v}}p(s){\rm d}s-p(\bar{v})V_{x}-\frac{p^\prime(\bar{v})V_{x}^{2}}{2}\right]{\rm d}x+C(\varepsilon +\delta) (1+t)^{-\frac{3}{2}} \|V_{x}(t)\|^{2}\\
  &+C\delta (1+t)^{-\frac{7}{2}}.
  \end{split}
  \end{equation}
Choosing $\lambda>0$ suitably small, we have from $\lambda \cdot \eqref{3.15}+\eqref{3.22}$ that
  \begin{align}\label{3.23}
    &\frac{1}{2}\frac{{\rm d}}{{\rm d}t}\int_{0}^{\infty} \left(V_{t}^{2}+\lambda V^{2}+2\lambda VV_{t}-p^\prime(\bar{v})V_{x}^{2}\right){\rm d} x+\frac{1}{2}\int_{0}^{\infty}(V_{t}^{2}-\lambda p^\prime(\bar{v})V_{x}^{2}){\rm d}x\nonumber\\
     \leq&\frac{{\rm d}}{{\rm d}t}\int_{0}^{\infty}\left[\int_{\bar{v}}^{V_{x}+\bar{v}+\hat{v}}p(s){\rm d}s-p(\bar{v})V_{x}-\frac{p^\prime(\bar{v})}{2}V_{x}^{2}\right]{\rm d}x+C\delta (1+t)^{-\frac{9}{4}}.
  \end{align}
  Integrating \eqref{3.23} respect to $t$ over $[0,t]$, we obtain
\begin{equation}\label{3.24}
    \|V(t)\|_{1}^{2}+\|V_{t}(t)\|^{2}+\int_{0}^{t}(\|V_{x}(s)\|^{2}+\|V_{t}(s)\|^{2}){\rm d}s\leq  C\left(\|V_{0}\|^{2}_{1}+\|z_{0}\|^{2}+ \delta \right).
  \end{equation}
  Multiplying \eqref{3.22} by $(1+t)$ and integrating by parts, one gets
 \begin{align}\label{3.25}
   &\frac{1}{2}\frac{{\rm d}}{{\rm d}t}\int_{0}^{\infty}(1+t)\left(V_{t}^{2}-p^\prime(\bar{v})V_{x}^{2}\right){\rm d} x+\frac{3}{4}\int_{0}^{\infty}(1+t)V_{t}^{2}{\rm d}x\nonumber\\
   \leq& \frac{{\rm d}}{{\rm d}t}(1+t)\int_{0}^{\infty}\left[\int_{\bar{v}}^{V_{x}+\bar{v}+\hat{v}}p(s){\rm d}s-p(\bar{v})V_{x}-\frac{p^\prime(\bar{v})}{2}V_{x}^{2}\right]{\rm d}x+C(\|V_{x}(t)\|^{2}+\|V_{t}(t)\|^{2})\nonumber\\
   &+C\delta (1+t)^{-\frac{5}{2}}.
  \end{align}
Integrating the above inequality in $t$ over $[0, t]$ and using \eqref{3.24}, one can immediately get \eqref{3.7}. This completes the proof of
Lemma \ref{l3.2}.
\end{proof}
\begin{lemma}\label{l3.3}
If $N(T) \leq \varepsilon^{2}$ and $\delta$ are sufficiently small, then it holds that
 \begin{equation}\label{3.26}
  \begin{split}
   (1+t)^{2}&(\|V_{xx}(t)\|^{2}+\|V_{xt}(t)\|^{2})+\int_{0}^{t}\left[(1+s)\|V_{xx}(s)\|^{2}+(1+s)^{2}\|V_{xt}(s)\|^{2}\right]{\rm d}s \\&
   \leq  C\left(\|V_{0}\|^{2}_{2}+\|z_{0}\|_{1}^{2}+ \delta\right),
   \end{split}
  \end{equation}
  for $0 \leq t \leq T$.
\end{lemma}
\begin{proof}
Multiplying $\eqref{1.16}_{1}$ by $-V_{xx}$ and integrating it with respect to $x$ over $\mathbb{R}^{+}$, we have
  \begin{equation}\label{3.27}
  \begin{split}
  \frac{1}{2}\frac{{\rm d}}{{\rm d}t}\int_{0}^{\infty} \left(V_{x}^{2}+2V_{x}V_{xt}\right){\rm d} x-\int_{0}^{\infty}p^\prime(\bar{v})V_{xx}^{2} {\rm d}x=&\int_{0}^{\infty}V_{xt}^{2}{\rm d} x+\int_{0}^{\infty}(p^\prime(\bar{v})_{x}V_{x}V_{xx}-F_{1}V_{xx}){\rm d}x\\
  &-\int_{0}^{\infty}F_{2}V_{xx}{\rm d}x.
  \end{split}
  \end{equation}
  We shall estimate the right hand side of \eqref{3.27} one by one. Firstly, by using \eqref{1.21}--\eqref{1.22}, \eqref{2.2}--\eqref{2.4} and \eqref{3.4}--\eqref{3.6b}, we have
  \begin{align}\label{3.28}
   &\int_{0}^{\infty}(p^\prime(\bar{v})_{x}V_{x}-F_{1})V_{xx}{\rm d}x \nonumber\\
   \leq& C\int_{0}^{\infty}|(p^\prime(V_{x}+\bar{v}+\hat{v})-p^\prime(\bar{v}))(V_{xx}+\bar{v}_{x})V_{xx}|{\rm d}x+C\int_{0}^{\infty}(|\hat{v}_{x}|+|\bar{v}_{x}\bar{v}_{t}|+|\bar{v}_{xt}|)|V_{xx}|{\rm d}x\nonumber\\
   &+C\int_{0}^{\infty}|\bar{v}_{x}V_{x}V_{xx}|{\rm d}x\nonumber\\
   \leq& -\frac{p^\prime(\bar{v})}{16}\|V_{xx}(t)\|^{2}+C\delta (1+t)^{-\frac{7}{2}}+C\delta(1+t)^{-2}\|V_{x}(t)\|^{2} ,
   \end{align}
   and
   \begin{align}\label{3.29}
   -\int_{0}^{\infty}F_{2}V_{xx}{\rm d}x=&-\int_{0}^{\infty}[g^\prime f(V_{xt}-p(\bar{v})_{xx}+\hat{v}_{t})+gf^\prime(V_{xx}+\bar{v}_{x}+\hat{v}_{x})]V_{xx}{\rm d}x \nonumber\\
  \leq& C\int_{0}^{\infty}|V_{xt}||V_{xx}|{\rm d}x+C(1+t)^{-1}\int_{0}^{\infty}(|\bar{v}_{xx}|+|\bar{v}_{x}|^{2}+|\hat{v}_{t}|)|V_{xx}|{\rm d}x \nonumber\\
  &+C(\varepsilon +\delta)\int_{0}^{\infty}V_{xx}^{2}{\rm d}x+C(1+t)^{-2}\int_{0}^{\infty}(|\bar{v}_{x}|+|\hat{v}_{x}|)|V_{xx}|{\rm d}x\nonumber\\
  \leq& -\frac{p^\prime(\bar{v})}{16}\|V_{xx}(t)\|^{2}+C\|V_{xt}(t)\|^{2}+C\delta (1+t)^{-\frac{9}{2}}.   	
       \end{align}
  Substituting \eqref{3.28}--\eqref{3.29} into \eqref{3.27}, we can get
\begin{equation}\label{3.30}
  \begin{split}
  \frac{1}{2}\frac{{\rm d}}{{\rm d}t}\int_{0}^{\infty} &\left(V_{x}^{2}+2V_{x}V_{xt}\right){\rm d}x-\frac{3}{4}\int_{0}^{\infty}p^\prime(\bar{v})V_{xx}^{2} {\rm d}x\\
  &\leq C\|V_{xt}(t)\|^{2}+C\delta (1+t)^{-\frac{7}{2}}+C\delta(1+t)^{-2}\|V_{x}(t)\|^{2}.
  \end{split}
  \end{equation}
  Next, the calculations of $\int_{\mathbb{R}^{+}}\eqref{1.16}_{1x}\times V_{xt}{\rm d}x$ gives
  \begin{equation}\label{3.31}
  \begin{split}
   \frac{1}{2}\frac{{\rm d}}{{\rm d}t}\int_{0}^{\infty}\left(V_{xt}^{2}-p^\prime(\bar{v})V_{xx}^{2}\right){\rm d}x+\int_{0}^{\infty}V_{xt}^{2}{\rm d}x=&\int_{0}^{\infty}\frac{p^{\prime\prime}(\bar{v})\bar{v}_{t}V_{xx}^{2}}{-2}{\rm d}x+\int_{0}^{\infty}\left(F_{1}-p^\prime(\bar{v})_{x}V_{x}\right)_{x}V_{xt}{\rm d}x\\
   &+\int_{0}^{\infty}F_{2x}V_{xt}{\rm d}x.
   \end{split}
  \end{equation}
  By using \eqref{2.2}--\eqref{2.4}, we have
  \begin{equation}\label{3.32}
 \int_{0}^{\infty}-\frac{p^{\prime\prime}(\bar{v})\bar{v}_{t}V_{xx}^{2}}{2}{\rm d}x \leq C\delta(1+t)^{-\frac{3}{2}}\|V_{xx}(t)\|^{2}.	
 \end{equation}
  From \eqref{1.21}--\eqref{1.22}, \eqref{2.2}--\eqref{2.4} and \eqref{3.4}--\eqref{3.6b}, we get
   \begin{equation}\label{3.33}
   \begin{split}
    &\int_{0}^{\infty}\left(F_{1}-p^\prime(\bar{v})_{x}V_{x}\right)_{x}V_{xt}{\rm d}x \\
     \leq& \frac{1}{16} \|V_{xt}(t)\|^{2}+\frac{1}{2}\frac{{\rm d}}{{\rm d}t}\int_{0}^{\infty}\left[p^\prime(V_{x}+\bar{v}+\hat{v})-p^\prime(\bar{v})\right]V_{xx}^{2}{\rm d}x+C\delta (1+t)^{-\frac{9}{2}}\\
    &+C(\varepsilon +\delta)(1+t)^{-\frac{3}{2}}\|V_{xx}(t)\|^{2}+C\delta(1+t)^{-3}\|V_{x}(t)\|^{2},
    \end{split}
  \end{equation}
  and
\begin{align}\label{3.34}
    &\int_{0}^{\infty}F_{2x}V_{xt}{\rm d}x\nonumber\\
    =&\int_{0}^{\infty}(g^\prime fV_{xt}+gf^\prime V_{xx})_{x}V_{xt}{\rm d}x+\int_{0}^{\infty}[g^\prime f(-p(\bar{v})_{xx}+\hat{v}_{t})]_{x}V_{xt}{\rm d}x+\int_{0}^{\infty}[gf^\prime(\bar{v}_{x}+\hat{v}_{x})]_{x}V_{xt}{\rm d}x\nonumber\\
    =&-\int_{0}^{\infty}(g^\prime fV_{xt}+gf^\prime V_{xx})V_{xxt}{\rm d}x+C(1+t)^{-1}\int_{0}^{\infty}(|\bar{v}_{xxx}|+|\bar{v}_{x}||\bar{v}_{xx}|+|\bar{v}_{x}|^{3}+|\hat{v}_{xt}|)|V_{xt}|{\rm d}x \nonumber\\
    &+C(1+t)^{-\frac{3}{2}}\int_{0}^{\infty}(|\bar{v}_{xx}|+|\bar{v}_{x}|^{2}+|\hat{v}_{t}|+|\hat{v}_{xx}|)|V_{xt}|{\rm d}x+C(1+t)^{-\frac{5}{2}}\int_{0}^{\infty}(|\bar{v}_{x}|+|\hat{v}_{x}|)|V_{xt}|{\rm d}x \nonumber\\
   \leq&\int_{0}^{\infty}\left(\frac{g^\prime f}{2}\right)_{x}V_{xt}^{2}{\rm d}x-\frac{{\rm d}}{{\rm d}t}\int_{0}^{\infty}\frac{gf^\prime}{2} V_{xx}^{2}{\rm d}x+\int_{0}^{\infty}\left(\frac{gf^\prime}{2}\right)_{t}V_{xx}^{2}{\rm d}x+\frac{1}{32}\int_{0}^{\infty}V_{xt}^{2}{\rm d}x+C\delta (1+t)^{-\frac{11}{2}} \nonumber\\
     \leq& \frac{1}{16}\|V_{xt}(t)\|^{2}+ C(\varepsilon +\delta)(1+t)^{-2}\|V_{xx}(t)\|^{2}-\frac{1}{2}\frac{{\rm d}}{{\rm d}t}\int_{0}^{\infty}gf^\prime V_{xx}^{2}{\rm d}x+C\delta (1+t)^{-\frac{11}{2}}.
  \end{align}
  Substituting \eqref{3.32}--\eqref{3.34} into \eqref{3.31}, we obtain
  \begin{align}\label{3.35}
    &\frac{1}{2}\frac{{\rm d}}{{\rm d}t}\int_{0}^{\infty}\left[V_{xt}^{2}+\left(gf^\prime-p^\prime(\bar{v})\right)V_{xx}^{2}\right]{\rm d}x+\frac{3}{4}\int_{0}^{\infty}V_{xt}^{2}{\rm d}x \nonumber\\
    \leq& \frac{1}{2}\frac{{\rm d}}{{\rm d}t}\int_{0}^{\infty}\left[p^\prime(V_{x}+\bar{v}+\hat{v})-p^\prime(\bar{v})\right]V_{xx}^{2}{\rm d}x+C(\varepsilon +\delta)(1+t)^{-\frac{3}{2}}\|V_{xx}(t)\|^{2}\nonumber\\
    &+C\delta(1+t)^{-3}\|V_{x}(t)\|^{2}+C\delta (1+t)^{-\frac{9}{2}}.
  \end{align}
  Addition of $\lambda \cdot \eqref{3.30}$ to \eqref{3.35} $(0<\lambda \ll 1)$, we can conclude that
  \begin{equation}\label{3.36}
   \begin{split}
   &\frac{1}{2}\frac{{\rm d}}{{\rm d}t}\int_{0}^{\infty} \left[V_{xt}^{2}+\lambda V_{x}^{2}+2\lambda V_{xt}V_{x}+\left(gf^\prime-p^\prime(\bar{v})\right)V_{xx}^{2}\right]{\rm d}x+\frac{1}{2}\int_{0}^{\infty} \left(V_{xt}^{2}-\lambda p^\prime(\bar{v})V_{xx}^{2}\right){\rm d} x \\
   \leq& \frac{1}{2}\frac{{\rm d}}{{\rm d}t}\int_{0}^{\infty}\left[p^\prime(V_{x}+\bar{v}+\hat{v})-p^\prime(\bar{v})\right]V_{xx}^{2}{\rm d}x+ C\delta(1+t)^{-\frac{7}{2}}+C\delta(1+t)^{-2}\|V_{x}(t)\|^{2}.
    \end{split}
  \end{equation}
   Integrating \eqref{3.36} over $[0,t]$ and using Lemma \ref{l3.2}, we get that
  \begin{equation}\label{3.37}
   \|V_{x}(t)\|_{1}^{2}+\|V_{xt}(t)\|^{2}+\int_{0}^{t}(\|V_{xx}(s)\|^{2}+\|V_{xt}(s)\|^{2}){\rm d}s
   \leq  C\left(\|V_{0}\|^{2}_{2}+\|z_{0}\|_{1}^{2}+ \delta \right).
  \end{equation}
Then multiplying \eqref{3.36} by $(1+t)$ and integrating it with respect to $t$, we obtain
\begin{equation}\label{3.38}
 \begin{split}
   (1+t)&(\|V_{x}(t)\|_{1}^{2}+\|V_{xt}(t)\|^{2})+\int_{0}^{t}\left[(1+s)(\|V_{xx}(s)\|^{2}+\|V_{xt}(s)\|^{2})\right]{\rm d}s \\&
   \leq  C\left(\|V_{0}\|^{2}_{2}+\|z_{0}\|_{1}^{2}+\delta \right).
   \end{split}
\end{equation}
Here we used Lemma \ref{l3.2} and \eqref{3.37}. Moreover, multiplying \eqref{3.35} by $(1+t)^{2}$ and integrating it over $[0,t]$ gives
\begin{equation}\label{3.39}
   (1+t)^{2}(\|V_{xx}(t)\|^{2}+\|V_{xt}(t)\|^{2})+\int_{0}^{t}(1+s)^{2}\|V_{xt}(s)\|^{2}{\rm d}s \leq  C\left(\|V_{0}\|^{2}_{2}+\|z_{0}\|_{1}^{2}+\delta \right).
 \end{equation}
 Here we used Lemma \ref{l3.2} and \eqref{3.38}.
 Combining two above equations, we can obtain \eqref{3.26}. The proof of this Lemma \ref{l3.3} is completed.
\end{proof}

\begin{lemma}\label{L3.3}
 If $N(T) \leq \varepsilon^{2}$ and $\delta$ are sufficiently small, then it holds that
\begin{equation}\label{3.40}
\begin{split}
   (1+t)^{3}&(\|V_{xxx}(t)\|^{2}+\|V_{xxt}(t)\|^{2})+\int_{0}^{t}\left[(1+s)^{2}\|V_{xxx}(s)\|^{2}+(1+s)^{3}\|V_{xxt}(s)\|^{2}\right]{\rm d}s \\&
   \leq  C\left(\|V_{0}\|^{2}_{3}+\|z_{0}\|_{2}^{2}+ \delta \right),
   \end{split}
  \end{equation}
  for $0 \leq t \leq T$.
\end{lemma}
\begin{proof}
By calculating $\int_{\mathbb{R}^{+}}\partial_{x}\eqref{1.16}_{1}\times (-V_{xxx}){\rm d}x$, we have after some integrations by parts that
\begin{align}\label{3.41}
  &\frac{{\rm d}}{{\rm d}t}\int_{0}^{\infty}(\frac{1}{2}V_{xx}^{2}+V_{xx}V_{xxt}){\rm d}x-\int_{0}^{\infty}p^\prime(\bar{v})V_{xxx}^{2}{\rm d}x \nonumber\\
  =&\int_{0}^{\infty}V_{xxt}^{2}{\rm d}x-\int_{0}^{\infty}\left(F_{1x}-p^\prime(\bar{v})_{xx}V_{x}-2p^\prime(\bar{v})_{x}V_{xx}\right)V_{xxx}{\rm d}x-\int_{0}^{\infty}F_{2x}V_{xxx}{\rm d}x.
  \end{align}
  We next estimate the right hand side of \eqref{3.41}. By using Lemmas \ref{l2.1}--\ref{l2.3} and \eqref{3.4}--\eqref{3.5}, we get
  \begin{align}\label{3.42}
  &-\int_{0}^{\infty}\left(F_{1x}-p^\prime(\bar{v})_{xx}V_{x}-2p^\prime(\bar{v})_{x}V_{xx}\right)V_{xxx}{\rm d}x \nonumber\\&
   \leq -\frac{p^\prime(\bar{v})}{16}\|V_{xxx}(t)\|^{2}+C\delta(1+t)^{-\frac{9}{2}}+C\delta(1+t)^{-3}\|V_{x}(t)\|^{2}+C(1+t)^{-2}\|V_{xx}(t)\|^{2}.
   \end{align}
   Noticing that
\begin{align}\label{3.43}
 -\int_{0}^{\infty}F_{2x}V_{xxx}{\rm d}x=&-\int_{0}^{\infty}\left(g^\prime fV_{xt}\right)_{x}V_{xxx}{\rm d}x-\int_{0}^{\infty}\left(gf^\prime V_{xx}\right)_{x}V_{xxx}{\rm d}x\nonumber\\
    &-\int_{0}^{\infty}[g^\prime f(-p(\bar{v})_{xx}+\hat{v}_{t})]_{x}V_{xxx}{\rm d}x-\int_{0}^{\infty}[gf^\prime(\bar{v}_{x}+\hat{v}_{x})]_{x}V_{xxx}{\rm d}x\nonumber\\
    &:= I_{6}+I_{7}+I_{8}+I_{9},
\end{align}
then by using \eqref{1.21}--\eqref{1.22}, \eqref{2.2}--\eqref{2.4} and \eqref{3.4}--\eqref{3.6b}, we obtain that
\begin{equation}\label{3.44}
    \begin{split}
   I_{6}=&-\int_{0}^{\infty}\left(g^\prime f\right)_{x}V_{xt}V_{xxx}{\rm d}x-\int_{0}^{\infty}g^\prime fV_{xxt}V_{xxx}{\rm d}x\\
   &\leq -\frac{p^\prime(\bar{v})}{16}\|V_{xxx}(t)\|^{2}+C\|V_{xxt}(t)\|^{2}+C(1+t)^{-3}\|V_{xt}(t)\|^{2},
    \end{split}
   \end{equation}
   \begin{equation}\label{3.45}
    I_{7}= \int_{0}^{\infty}\left(gf^\prime\right)_{x}V_{xx}V_{xxx}{\rm d}x+\int_{0}^{\infty}gf^\prime V_{xxx}^{2}{\rm d}x
    \leq -\frac{p^\prime(\bar{v})}{16}\|V_{xxx}(t)\|^{2}+C(1+t)^{-5}\|V_{xx}(t)\|^{2},
   \end{equation}
   and
   \begin{equation}\label{3.46}
    \begin{split}
    I_{8}+I_{9} &\leq C(1+t)^{-1}\int_{0}^{\infty}(|\bar{v}_{xxx}|+|\bar{v}_{x}||\bar{v}_{xx}|+|\bar{v}_{x}|^{3}+|\hat{v}_{xt}|)|V_{xxx}|{\rm d}x\\
    &~~~~+C(1+t)^{-\frac{5}{2}}\int_{0}^{\infty}(|\bar{v}_{x}|+|\hat{v}_{x}|)|V_{xxx}|{\rm d}x\\
    &~~~~+C(1+t)^{-\frac{3}{2}}\int_{0}^{\infty}(|\bar{v}_{xx}|+|\bar{v}_{x}|^{2}+|\hat{v}_{t}|+|\hat{v}_{xx}|)|V_{xxx}|{\rm d}x\\
    &\leq -\frac{p^\prime(\bar{v})}{16}\|V_{xxx}(t)\|^{2}+C\delta (1+t)^{-\frac{11}{2}}.
    \end{split}
   \end{equation}
  Substituting \eqref{3.42}--\eqref{3.46} into \eqref{3.41}, we have
  \begin{equation}\label{3.47}
  \begin{split}
     &\frac{1}{2}\frac{{\rm d}}{{\rm d}t}\int_{0}^{\infty}(V_{xx}^{2}+V_{xx}V_{xxt}){\rm d}x-\frac{3}{4}\int_{0}^{\infty}p^\prime(\bar{v})V_{xxx}^{2}{\rm d}x\\
     \leq&  C\|V_{xxt}(t)\|^{2}+C(1+t)^{-3}(\|V_{x}(t)\|^{2}+\|V_{xt}(t)\|^{2})+C(1+t)^{-2}\|V_{xx}(t)\|^{2}+C\delta (1+t)^{-\frac{9}{2}}.
     \end{split}
  \end{equation}
Then we have from $\int_{\mathbb{R}^{+}}\partial_{x}^{2}\eqref{1.16}_{1}\times V_{xxt}{\rm d}x$ that
\begin{align}\label{3.48}
&\frac{1}{2}\frac{{\rm d}}{{\rm d}t}\int_{0}^{\infty}\left(V_{xxt}^{2}-p^\prime(\bar{v})V_{xxx}^{2}\right){\rm d}x+\int_{0}^{\infty}V_{xxt}^{2}{\rm d}x\nonumber\\
   =&-\frac{1}{2}\int_{0}^{\infty}p^{\prime\prime}(\bar{v})\bar{v}_{t}V_{xxx}^{2}{\rm d}x+ \int_{0}^{\infty}\left[F_{1x}-p^\prime(\bar{v})_{xx}V_{x}-2p^\prime(\bar{v})_{x}V_{xx}\right]_{x}V_{xxt}{\rm d}x\nonumber\\
   &+\int_{0}^{\infty}F_{2xx}V_{xxt}{\rm d}x.	
\end{align}
By using Lemmas \ref{l2.1}--\ref{l2.3} and \eqref{3.4}--\eqref{3.5}, we can deduce that
\begin{equation}\label{3.49}
-\frac{1}{2}\int_{0}^{\infty}p^{\prime\prime}(\bar{v})\bar{v}_{t}V_{xxx}^{2}{\rm d}x\leq C\delta(1+t)^{-\frac{3}{2}}\|V_{xxx}(t)\|^{2},	
\end{equation}
and
\begin{align}\label{3.50}
&\int_{0}^{\infty}\left[F_{1x}-p^\prime(\bar{v})_{xx}V_{x}-2p^\prime(\bar{v})_{x}V_{xx}\right]_{x}V_{xxt}{\rm d}x\nonumber\\
\leq& \frac{1}{16} \|V_{xxt}(t)\|^{2}+\frac{1}{2}\frac{{\rm d}}{{\rm d}t}\int_{0}^{\infty}\left[p^\prime(V_{x}+\bar{v}+\hat{v})-p^\prime(\bar{v})\right]V_{xxx}^{2}{\rm d}x+C\delta(1+t)^{-\frac{11}{2}}+C(1+t)^{-4}\|V_{x}(t)\|^{2}\nonumber\\
  &+C(\varepsilon +\delta)(1+t)^{-\frac{3}{2}}\|V_{xxx}(t)\|^{2}+C(1+t)^{-3}\|V_{xx}(t)\|^{2}.	
\end{align}
Noticing that
\begin{align}\label{3.51}
\int_{0}^{\infty}F_{2xx}V_{xxt}{\rm d}x=&\int_{0}^{\infty}(g^\prime fV_{xt})_{xx}V_{xxt}{\rm d}x+\int_{0}^{\infty}(gf^\prime V_{xx})_{xx}V_{xxt}{\rm d}x\nonumber\\
&+\int_{0}^{\infty}[g^\prime f(-p(\bar{v})_{xx}+\hat{v}_{t})]_{xx}V_{xxt}{\rm d}x\nonumber\\
&+\int_{0}^{\infty}[gf^\prime(\bar{v}_{x}+\hat{v}_{x})]_{xx}V_{xxt}{\rm d}x:=I_{10}+I_{11}+I_{12}+I_{13}.	
\end{align}
From \eqref{1.21}--\eqref{1.22}, \eqref{2.2}--\eqref{2.4} and \eqref{3.4}--\eqref{3.6b}, we can conclude that
\begin{align}\label{3.52}
I_{10}&=\int_{0}^{\infty}\left(g^\prime f\right)_{xx}V_{xt}V_{xxt}{\rm d}x+2\int_{0}^{\infty}\left(g^\prime f\right)_{x}V_{xxt}^{2}{\rm d}x+\int_{0}^{\infty}g^\prime fV_{xxt}V_{xxxt}{\rm d}x\nonumber\\
&=\int_{0}^{\infty}\left[\left(g^\prime f\right)_{xx}V_{xt}V_{xxt}+\frac{3}{2}\left(g^\prime f\right)_{x}V_{xxt}^{2}\right]{\rm d}x\nonumber\\
&\leq \frac{1}{32}\int_{0}^{\infty}V_{xxt}^{2}{\rm d}x+C(1+t)^{-4}\int_{0}^{\infty}V_{xt}^{2}{\rm d}x+C\|(uV_{xt})(t)\|_{L^{\infty}}\int_{0}^{\infty}|V_{xxt}||V_{xxx}|{\rm d}x\nonumber\\
    &\leq \frac{1}{16} \|V_{xxt}(t)\|^{2}+C(1+t)^{-4}\|V_{xt}(t)\|^{2}+C(\varepsilon+\delta)(1+t)^{-\frac{15}{4}}\|V_{xxx}(t)\|^{2}, 		
\end{align}
\begin{align}\label{3.53}
I_{11}&=\int_{0}^{\infty}[(gf^\prime)_{xx}V_{xx}+2(gf^\prime)_{x}V_{xxx}]V_{xxt}{\rm d}x+\int_{0}^{\infty}gf^\prime V_{xxxx}V_{xxt}{\rm d}x\nonumber\\
&=\int_{0}^{\infty}\{[(gf^\prime)_{xx}V_{xx}+(gf^\prime)_{x}V_{xxx}]V_{xxt}-gf^\prime V_{xxx}V_{xxxt}\}{\rm d}x\nonumber\\
&\leq \frac{1}{32}\int_{0}^{\infty}V_{xxt}^{2}{\rm d}x+C(1+t)^{-6}\int_{0}^{\infty}V_{xx}^{2}{\rm d}x+C\|(u^{2}V_{xx})(t)\|_{L^{\infty}}\int_{0}^{\infty}|V_{xxt}||V_{xxx}|{\rm d}x\nonumber\\
    &~~+C(\varepsilon +\delta)(1+t)^{-5}\int_{0}^{\infty}V_{xxx}^{2}{\rm d}x-\frac{{\rm d}}{{\rm d}t}\int_{0}^{\infty}\frac{gf^\prime}{2}V_{xxx}^{2}{\rm d}x+\int_{0}^{\infty}\left(\frac{gf^\prime}{2}\right)_{t}V_{xxx}^{2}{\rm d}x\nonumber\\
    &\leq \frac{1}{16} \|V_{xxt}(t)\|^{2}+C(1+t)^{-6}\|V_{xx}(t)\|^{2}+C(\varepsilon +\delta)(1+t)^{-3}\|V_{xxx}(t)\|^{2}\nonumber\\
    &~~-\frac{1}{2}\frac{{\rm d}}{{\rm d}t}\int_{0}^{\infty}gf^\prime V_{xxx}^{2}{\rm d}x,	
\end{align}
and
\begin{equation}\label{3.54}
I_{12}+I_{13} \leq \frac{1}{16} \|V_{xxt}(t)\|^{2}+C\delta(1+t)^{-5}\|V_{xxx}(t)\|^{2}+C\delta(1+t)^{-\frac{13}{2}}.	
\end{equation}
Substituting \eqref{3.49}--\eqref{3.54} into \eqref{3.48}, we have
\begin{equation}\label{3.55}
\begin{split}
    &\frac{1}{2}\frac{{\rm d}}{{\rm d}t}\int_{0}^{\infty}\left[V_{xxt}^{2}+\left(gf^\prime-p^\prime(\bar{v})\right)V_{xxx}^{2}\right]{\rm d}x+\frac{3}{4}\int_{0}^{\infty}V_{xxt}^{2}{\rm d}x \\
    \leq& \frac{1}{2}\frac{{\rm d}}{{\rm d}t}\int_{0}^{\infty}\left[p^\prime(V_{x}+\bar{v}+\hat{v})-p^\prime(\bar{v})\right]V_{xxx}^{2}{\rm d}x+C(\varepsilon +\delta)(1+t)^{-\frac{3}{2}}\|V_{xxx}(t)\|^{2}+C\delta (1+t)^{-\frac{11}{2}}\\&
    +C(1+t)^{-3}\|V_{xx}(t)\|^{2}+C(1+t)^{-4}(\|V_{x}(t)\|^{2}+\|V_{xt}(t)\|^{2}).
    \end{split}
\end{equation}
  Addition of $\lambda \cdot \eqref{3.47}$ to \eqref{3.55} $(0<\lambda \ll 1)$, we have
  \begin{equation}\label{3.56}
   \begin{split}
   &\frac{1}{2}\frac{{\rm d}}{{\rm d}t}\int_{0}^{\infty} \left[V_{xxt}^{2}+\lambda V_{xx}^{2}+2\lambda V_{xxt}V_{xx}+\left(gf^\prime-p^\prime(\bar{v})\right)V_{xxx}^{2}\right]{\rm d}x+\frac{1}{2}\int_{0}^{\infty} \left(V_{xxt}^{2}-\lambda p^\prime(\bar{v})V_{xxx}^{2}\right){\rm d} x \\
   \leq& \frac{1}{2}\frac{{\rm d}}{{\rm d}t}\int_{0}^{\infty}\left[p^\prime(V_{x}+\bar{v}+\hat{v})-p^\prime(\bar{v})\right]V_{xxx}^{2}{\rm d}x+ C\delta (1+t)^{-\frac{9}{2}}+C(1+t)^{-2}\|V_{xx}(t)\|^{2}\\&
   +C(1+t)^{-3}(\|V_{x}(t)\|^{2}+\|V_{xt}(t)\|^{2}).
    \end{split}
    \end{equation}
  Integrating the above equation over $[0,t]$ and using Lemmas \ref{l3.2}--\ref{l3.3}, we obtain
\begin{equation}\label{3.57}
   \|V_{xx}(t)\|_{1}^{2}+\|V_{xxt}(t)\|^{2}+\int_{0}^{t}(\|V_{xxx}(s)\|^{2}+\|V_{xxt}(s)\|^{2}){\rm d}s\leq  C\left(\|V_{0}\|^{2}_{3}+\|z_{0}\|_{2}^{2}+ \delta \right).
  \end{equation}
Integrating \eqref{3.56}, $(1+t)\cdot \eqref{3.56}$ and $(1+t)^{2}\cdot \eqref{3.56}$ over $[0,t]$ respectively, one gets
\begin{equation}\label{3.58}
   \begin{split}
   (1+t)^{2}(\|V_{xx}(t)\|_{1}^{2}+&\|V_{xxt}(t)\|^{2})+\int_{0}^{t}(1+s)^{2}(\|V_{xxx}(s)\|^{2}+\|V_{xxt}(s)\|^{2}){\rm d}s \\&
   \leq C\left(\|V_{0}\|^{2}_{3}+\|z_{0}\|_{2}^{2}+ \delta \right) .
    \end{split}
  \end{equation}
  Then integrating $(1+t)^{3} \cdot \eqref{3.55}$ over $[0,t]$ gives
  \begin{equation}\label{3.59}
   (1+t)^{3}(\|V_{xxx}(t)\|^{2}+\|V_{xxt}(t)\|^{2})+\int_{0}^{t}(1+s)^{3}\|V_{xxt}(s)\|^{2}{\rm d}s
   \leq  C\left(\|V_{0}\|^{2}_{3}+\|z_{0}\|_{2}^{2}+\delta \right) .
  \end{equation}
Here we have used Lemmas \ref{l3.2}--\ref{l3.3} and \eqref{3.58}. Combining the above two equations, one can immediately obtain \eqref{3.40}. The proof of this Lemma is completed.
\end{proof}
\begin{lemma}\label{l3.5}
 If $N(T) \leq \varepsilon^{2}$ and $\delta$ are sufficiently small, then it holds that
 \begin{equation}\label{3.60}
 \begin{split}
   (1+t)^{2}\|V_{t}(t)\|^{2}&+(1+t)^{3}(\|V_{xt}(t)\|^{2}+\|V_{tt}(t)\|^{2})\\
   &+\int_{0}^{t}\left[(1+s)^{2}\|V_{xt}(s)\|^{2}+(1+s)^{3}\|V_{tt}(s)\|^{2}\right]{\rm d}s \\
   \leq&  C\left(\|V_{0}\|^{2}_{2}+\|z_{0}\|_{1}^{2}+ \delta \right),
   \end{split}
  \end{equation}
  for $0 \leq t \leq T$.
\end{lemma}
\begin{proof}
Firstly, by calculating $\int_{\mathbb{R}^{+}}\partial_{t}\eqref{1.16}_{1}\times V_{t}{\rm d}x$, we have
\begin{equation}\label{3.61}
\begin{split}
\frac{1}{2}\frac{{\rm d}}{{\rm d}t}&\int_{0}^{\infty} \left(V_{t}^{2}+2V_{t}V_{tt}\right){\rm d} x-\int_{0}^{\infty} p^\prime(\bar{v})V_{xt}^{2}{\rm d} x \\
&=\int_{0}^{\infty}V_{tt}^{2}{\rm d}x+\int_{0}^{\infty}(F_{1t}V_{t}+p^\prime(\bar{v})_{t}V_{x}V_{xt}){\rm d}x+\int_{\mathbb{R}}F_{2t}V_{t}{\rm d}x.
\end{split}
\end{equation}
  By using Lemmas \ref{l2.1}--\ref{l2.3} and \eqref{3.4}--\eqref{3.5}, one gets
  \begin{equation}\label{3.62}
  \begin{split}
   &\int_{0}^{\infty}(F_{1t}V_{t}+p^\prime(\bar{v})_{t}V_{x}V_{xt}){\rm d}x\\&
   \leq -\frac{p^\prime(\bar{v})}{16}\|V_{xt}(t)\|^{2}+C\delta (1+t)^{-\frac{9}{2}}+C\delta(1+t)^{-3}\|V_{x}(t)\|^{2}.
   \end{split}
  \end{equation}
  Next, notice that
  \begin{align}\label{3.63}
    \int_{0}^{\infty}F_{2t}V_{t}{\rm d}x=&\int_{0}^{\infty}\left(g^\prime fV_{xt}\right)_{t}V_{t}{\rm d}x+\int_{0}^{\infty}\left(gf^\prime V_{xx}\right)_{t}V_{t}{\rm d}x\nonumber\\
    &+\int_{0}^{\infty}[g^\prime f(-p(\bar{v})_{xx}+\hat{v}_{t})]_{t}V_{t}{\rm d}x\nonumber\\
    &+\int_{0}^{\infty}[gf^\prime(\bar{v}_{x}+\hat{v}_{x})]_{t}V_{t}{\rm d}x:= I_{14}+I_{15}+I_{16}+I_{17}.
    \end{align}
  By employing \eqref{1.21}--\eqref{1.22}, \eqref{2.2}--\eqref{2.4} and \eqref{3.4}--\eqref{3.6b}, one has
  \begin{equation}\label{3.64}
    \begin{split}
    I_{14}&=\int_{0}^{\infty}\left(g^\prime f\right)_{t}V_{xt}V_{t}{\rm d}x+\int_{0}^{\infty}g^\prime fV_{xtt}V_{t}{\rm d}x\\
    &\leq -\frac{p^\prime(\bar{v})}{32}\|V_{xt}(t)\|^{2}+C(1+t)^{-4}\|V_{t}(t)\|^{2}-\int_{0}^{\infty}g^\prime fV_{tt}V_{xt}{\rm d}x-\int_{0}^{\infty}\left(g^\prime f\right)_{x}V_{tt}V_{t}{\rm d}x\\
    &\leq -\frac{p^\prime(\bar{v})}{16}\|V_{xt}(t)\|^{2}+C(1+t)^{-3}\|V_{t}(t)\|^{2}+C\|V_{tt}(t)\|^{2},
    \end{split}
  \end{equation}
  \begin{equation}\label{3.65}
    \begin{split}
    I_{15}&=\int_{0}^{\infty}\left(gf^\prime\right)_{t}V_{xx}V_{t}{\rm d}x+\int_{0}^{\infty}gf^\prime V_{xxt}V_{t}{\rm d}x\\
    &\leq C(1+t)^{-3}(\|V_{t}(t)\|^{2}+\|V_{xx}(t)\|^{2})-\int_{0}^{\infty}gf^\prime V_{xt}^{2}{\rm d}x-\int_{0}^{\infty}\left(gf^\prime\right)_{x}V_{xt}V_{t}{\rm d}x\\
    &\leq -\frac{p^\prime(\bar{v})}{16}\|V_{xt}(t)\|^{2}+C(1+t)^{-3}(\|V_{t}(t)\|^{2}+\|V_{xx}(t)\|^{2}).
    \end{split}
  \end{equation}
 Similarly, we can prove
  \begin{equation}\label{3.66}
    I_{16}+I_{17} \leq C(1+t)^{-1}\|V_{t}(t)\|^{2}+C\|V_{tt}(t)\|^{2}+C\delta(1+t)^{-\frac{11}{2}}.
  \end{equation}
  Substituting \eqref{3.62}--\eqref{3.66} into \eqref{3.61}, we can obtain
  \begin{align}\label{3.67}
     &\frac{1}{2}\frac{{\rm d}}{{\rm d}t}\int_{0}^{\infty}(V_{t}^{2}+2V_{t}V_{tt}){\rm d}x-\frac{3}{4}\int_{0}^{\infty}p^\prime(\bar{v})V_{xt}^{2}{\rm d}x\nonumber\\
     \leq&C\|V_{tt}(t)\|^{2}+C(1+t)^{-3}\|V_{x}(t)\|_{1}^{2}+C\delta(1+t)^{-\frac{9}{2}}+C(1+t)^{-1}\|V_{t}(t)\|^{2}.
     \end{align}
Then by calculating $\int_{\mathbb{R}^{+}}\partial_{t}\eqref{1.16}_{1}\times V_{tt}{\rm d}x$, we obtain
\begin{equation}\label{3.68}
\begin{split}
  \frac{1}{2}\frac{{\rm d}}{{\rm d}t}\int_{0}^{\infty}&(V_{tt}^{2}-p^\prime(\bar{v})V_{xt}^{2}){\rm d}x +\int_{0}^{\infty}V_{tt}^{2}{\rm d}x \\
  & = -\int_{0}^{\infty}\frac{p^{\prime\prime}(\bar{v})\bar{v}_{t}}{2}V_{xt}^{2}{\rm d}x+\int_{0}^{\infty}\left[F_{1t}+\left(p^\prime(\bar{v})_{t}V_{x}\right)_{x}\right]V_{tt}{\rm d}x+\int_{0}^{\infty}F_{2t}V_{tt}{\rm d}x .
  \end{split}
  \end{equation}
  The right hand side of \eqref{3.68} can be estimated as follows. From Lemmas \ref{l2.1}--\ref{l2.3} and \eqref{3.4}--\eqref{3.5}, we have
  \begin{equation}\label{3.69}
  -\int_{0}^{\infty}\frac{p^{\prime\prime}(\bar{v})\bar{v}_{t}}{2}V_{xt}^{2}{\rm d}x \leq C\delta(1+t)^{-\frac{3}{2}}\|V_{xt}(t)\|^{2},
  \end{equation}
  and
  \begin{align}\label{3.70}
   &\int_{0}^{\infty}\left[F_{1t}+\left(p^\prime(\bar{v})_{t}V_{x}\right)_{x}\right]V_{tt}{\rm d}x \nonumber\\
   \leq& \frac{1}{16}\|V_{tt}(t)\|^{2}+\frac{1}{2}\frac{{\rm d}}{{\rm d}t}\int_{0}^{\infty}\left(p^\prime(V_{x}+\bar{v}+\hat{v})-p^\prime(\bar{v})\right)V_{xt}^{2}{\rm d}x+C\delta (1+t)^{-\frac{11}{2}}+C\delta(1+t)^{-4}\|V_{x}(t)\|^{2}\nonumber\\
   &+C(\delta+\varepsilon)(1+t)^{-\frac{3}{2}}\|V_{xt}(t)\|^{2}+C(1+t)^{-3}\|V_{xx}(t)\|^{2}.
   \end{align}
  Finally, notice that
  \begin{equation}\label{3.71}
    \begin{split}
    \int_{0}^{\infty}F_{2t}V_{tt}{\rm d}x=&\int_{0}^{\infty}\left(g^\prime fV_{xt}\right)_{t}V_{tt}{\rm d}x+\int_{0}^{\infty}\left(gf^\prime V_{xx}\right)_{t}V_{tt}{\rm d}x\\
    &+\int_{0}^{\infty}[g^\prime f(-p(\bar{v})_{xx}+\hat{v}_{t})]_{t}V_{tt}{\rm d}x\\
    &+\int_{0}^{\infty}[gf^\prime(\bar{v}_{x}+\hat{v}_{x})]_{t}V_{tt}{\rm d}x:= I_{18}+I_{19}+I_{20}+I_{21}.
    \end{split}
  \end{equation}
  By employing \eqref{1.21}--\eqref{1.22}, \eqref{2.2}--\eqref{2.4} and \eqref{3.4}--\eqref{3.6b}, we obtain
  \begin{equation}\label{3.72}
   I_{18}=\int_{0}^{\infty}\left(g^\prime f\right)_{t}V_{xt}V_{tt}{\rm d}x+\int_{0}^{\infty}g^\prime fV_{xtt}V_{tt}{\rm d}x\leq \frac{1}{16} \|V_{tt}(t)\|^{2}+C(\varepsilon +\delta)(1+t)^{-4}\|V_{xt}(t)\|^{2},
   \end{equation}
   \begin{equation}\label{3.73}
    \begin{split}
    I_{19}&=\int_{0}^{\infty}\left(gf^\prime\right)_{t}V_{xx}V_{tt}{\rm d}x+\int_{0}^{\infty}gf^\prime V_{xxt}V_{tt}{\rm d}x\\
    &\leq \frac{1}{32} \|V_{tt}(t)\|^{2}+C(1+t)^{-6}\|V_{xx}(t)\|^{2}-\frac{1}{2}\frac{{\rm d}}{{\rm d}t}\int_{0}^{\infty}gf^\prime V_{xt}^{2}{\rm d}x+\frac{1}{2}\int_{0}^{\infty}\left(gf^\prime\right)_{t}V_{xt}^{2}{\rm d}x\\
    &~~~~-\int_{0}^{\infty}\left(gf^\prime\right)_{x}V_{xt}V_{tt}{\rm d}x\\
    &\leq \frac{1}{16} \|V_{tt}(t)\|^{2}+C(1+t)^{-6}\|V_{xx}(t)\|^{2}+C(\varepsilon +\delta)(1+t)^{-3}\|V_{xt}(t)\|^{2}-\frac{1}{2}\frac{{\rm d}}{{\rm d}t}\int_{0}^{\infty}gf^\prime V_{xt}^{2}{\rm d}x,
      \end{split}
    \end{equation}
  and
  \begin{equation}\label{3.74}
  \begin{split}
  I_{20}+I_{21} &\leq C(1+t)^{-1}\int_{0}^{\infty}(|\bar{v}_{xxt}|+|\bar{v}_{xx}||\bar{v}_{t}|+|\bar{v}_{x}||\bar{v}_{xt}|+|\bar{v}_{x}|^{2}|\bar{v}_{t}|+|\hat{v}_{tt}|)|V_{tt}|{\rm d}x\\
  &~~~~+C(1+t)^{-2}\int_{0}^{\infty}(|\bar{v}_{xx}|+|\bar{v}_{x}|^{2}+|\hat{v}_{t}|)|V_{tt}|{\rm d}x+C(1+t)^{-2}\int_{0}^{\infty}(|\bar{v}_{xt}|+|\hat{v}_{xt}|)|V_{tt}|{\rm d}x\\
  &~~~~+C(1+t)^{-3}\int_{0}^{\infty}(|\bar{v}_{x}|+|\hat{v}_{x}|)|V_{tt}|{\rm d}x\\
  &\leq \frac{1}{16} \|V_{tt}(t)\|^{2}+C\delta(1+t)^{-\frac{13}{2}}.
  \end{split}
  \end{equation}
Substituting \eqref{3.69}--\eqref{3.74} into \eqref{3.68}, we obtain
\begin{align}\label{3.75}
    &\frac{1}{2}\frac{{\rm d}}{{\rm d}t}\int_{0}^{\infty}\left[V_{tt}^{2}+\left(gf^\prime-p^\prime(\bar{v})\right)V_{xt}^{2}\right]{\rm d}x+\frac{3}{4}\int_{0}^{\infty}V_{tt}^{2}{\rm d}x \nonumber\\
    \leq& \frac{1}{2}\frac{{\rm d}}{{\rm d}t}\int_{0}^{\infty}\left[p^\prime(V_{x}+\bar{v}+\hat{v})-p^\prime(\bar{v})\right]V_{xt}^{2}{\rm d}x+C(\varepsilon +\delta)(1+t)^{-\frac{3}{2}}\|V_{xt}(t)\|^{2}+C\delta (1+t)^{-\frac{11}{2}}\nonumber\\
    &+C(1+t)^{-3}\|V_{xx}(t)\|^{2}+C(1+t)^{-4}\|V_{x}(t)\|^{2}.
    \end{align}
  Addition of $\lambda \cdot \eqref{3.67}$ to \eqref{3.75} $(0<\lambda \ll 1)$, we have
   \begin{align}\label{3.76}
   &\frac{1}{2}\frac{{\rm d}}{{\rm d}t}\int_{0}^{\infty} \left[V_{tt}^{2}+\lambda V_{t}^{2}+2\lambda V_{t}V_{tt}+\left(gf^\prime-p^\prime(\bar{v})\right)V_{xt}^{2}\right]{\rm d}x+\frac{1}{2}\int_{0}^{\infty}\left(V_{tt}^{2}-\lambda p^\prime(\bar{v})V_{xt}^{2}\right){\rm d} x \nonumber\\
   \leq& \frac{1}{2}\frac{{\rm d}}{{\rm d}t}\int_{0}^{\infty}\left[p^\prime(V_{x}+\bar{v}+\hat{v})-p^\prime(\bar{v})\right]V_{xt}^{2}{\rm d}x+ C\delta (1+t)^{-\frac{9}{2}}+C(1+t)^{-1}\|V_{t}(t)\|^{2}\nonumber\\
   &+C(1+t)^{-3}\|V_{x}(t)\|_{1}^{2}.
    \end{align}
  Integrating \eqref{3.76}, $(1+t)\cdot \eqref{3.76}$ and $(1+t)^{2}\cdot \eqref{3.76}$ over $[0,t]$, we obtain
\begin{equation}\label{3.77}
   (1+t)^{2}(\|V_{t}(t)\|_{1}^{2}+\|V_{tt}(t)\|^{2})+\int_{0}^{t}(1+s)^{2}(\|V_{tt}(s)\|^{2}+\|V_{xt}(s)\|^{2}){\rm d}s
   \leq  C\left(\|V_{0}\|^{2}_{3}+\|z_{0}\|_{2}^{2}+ \delta \right).
  \end{equation}
  Then the integration of $(1+t)^{3} \cdot \eqref{3.75}$ over $[0,t]$ yields
  \begin{equation}\label{3.78}
   (1+t)^{3}(\|V_{xt}(t)\|^{2}+\|V_{tt}(t)\|^{2})+\int_{0}^{t}(1+s)^{3}\|V_{tt}(s)\|^{2}{\rm d}s
   \leq  C\left(\|V_{0}\|^{2}_{3}+\|z_{0}\|_{2}^{2}+\delta \right) .
  \end{equation}
  Combining \eqref{3.77} and \eqref{3.78}, \eqref{3.60} is immediately obtained, then the proof of Lemma \ref{l3.5} is completed.
\end{proof}

\begin{lemma}\label{l3.6}
 If $N(T) \leq \varepsilon^{2}$ and $\delta$ are sufficiently small, then it holds that
 \begin{equation}\label{3.79}
 \begin{split}
   (1+t)^{4}&(\|V_{xtt}(t)\|^{2}+\|V_{xxt}(t)\|^{2})+\int_{0}^{t}\left[(1+s)^{3}\|V_{xxt}(s)\|^{2}+(1+s)^{4}\|V_{xtt}(s)\|^{2}\right]{\rm d}s \\&
   \leq  C\left(\|V_{0}\|^{2}_{3}+\|z_{0}\|_{2}^{2}+ \delta \right),
   \end{split}
  \end{equation}
  for $0 \leq t \leq T$.
\end{lemma}
\begin{proof}
  By calculating $\int_{\mathbb{R}^{+}}\partial_{xt}\eqref{1.16}_{1}\times V_{xtt}{\rm d}x$, we get
\begin{equation}\label{3.80}
\begin{split}
 &\frac{1}{2}\frac{{\rm d}}{{\rm d}t}\int_{0}^{\infty}(V_{xtt}^{2}-p^\prime(\bar{v})V_{xxt}^{2}){\rm d}x +\int_{0}^{\infty}V_{xtt}^{2}{\rm d}x \\
  =&-\frac{1}{2}\int_{0}^{\infty}p^\prime(\bar{v})_{t}V_{xxt}^{2}{\rm d}x+\int_{0}^{\infty}\left[F_{1xt}-\left(p^\prime(\bar{v})_{t}V_{xx}+(p^\prime(\bar{v})_{x}V_{x})_{t}\right)_{x}\right]V_{xtt}{\rm d}x+\int_{0}^{\infty}F_{2xt}V_{xtt}{\rm d}x.
  \end{split}
  \end{equation}
  We estimate the right hand side of \eqref{3.80} as follows. From Lemmas \ref{l2.1}--\ref{l2.3} and \eqref{3.4}--\eqref{3.5}, one gets
  \begin{equation}\label{3.81}
   -\frac{1}{2}\int_{0}^{\infty}p^\prime(\bar{v})_{t}V_{xxt}^{2}{\rm d}x \leq C\delta(1+t)^{-\frac{3}{2}}\|V_{xxt}(t)\|^{2},	
   \end{equation}
  and
  \begin{equation}\label{3.82}
  \begin{split}
   &\int_{0}^{\infty}\left[F_{1xt}-\left(p^\prime(\bar{v})_{t}V_{xx}+(p^\prime(\bar{v})_{x}V_{x})_{t}\right)_{x}\right]V_{xtt}{\rm d}x \\
   \leq& \frac{1}{16}\|V_{xtt}(t)\|^{2}+\frac{1}{2}\frac{{\rm d}}{{\rm d}t}\int_{0}^{\infty}\left[p^\prime(V_{x}+\bar{v}+\hat{v})-p^\prime(\bar{v})\right]V_{xxt}^{2}{\rm d}x+C\delta (1+t)^{-\frac{13}{2}}+C(1+t)^{-3}\|V_{xt}(t)\|^{2}\\&
   +C(1+t)^{-4}\|V_{xx}(t)\|^{2}+C(1+t)^{-3}\|V_{xxx}(t)\|^{2}+C(\delta+\varepsilon)(1+t)^{-\frac{3}{2}}\|V_{xxt}(t)\|^{2}\\&
   +C(1+t)^{-5}\|V_{x}(t)\|^{2}.
   \end{split}
  \end{equation}
  Now we turn to deal with the last term of the righthand side of \eqref{3.80}. Notice that
  \begin{equation}\label{3.83}
  	\begin{split}
  	\int_{0}^{\infty}F_{2xt}V_{xtt}{\rm d}x&=\int_{0}^{\infty}\left(g^\prime fV_{xt}\right)_{xt}V_{xtt}{\rm d}x+\int_{0}^{\infty}\left(gf^\prime V_{xx}\right)_{xt}V_{xtt}{\rm d}x\\
&+\int_{0}^{\infty}(g^\prime f(-p(\bar{v})_{xt}+\hat{v}_{t}))_{xt}V_{xtt}{\rm d}x+\int_{0}^{\infty}(gf^\prime(\bar{v}_{x}+\hat{v}_{x}))_{xt}V_{xtt}{\rm d}x\\
&:=I_{22}+I_{23}+I_{24}+I_{25}.	
  	\end{split}
  \end{equation}
  It is easy to get that
  \begin{equation}\label{3.84}
  	\begin{split}
  	I_{22}&=\int_{0}^{\infty}\left(g^\prime f\right)_{xt}V_{xt}V_{xtt}{\rm d}x+\int_{0}^{\infty}\left(g^\prime f\right)_{x}V_{xtt}^{2}{\rm d}x+\int_{0}^{\infty}\left(g^\prime f\right)_{t}V_{xxt}V_{xtt}{\rm d}x+\int_{0}^{\infty}g^\prime fV_{xxtt}V_{xtt}{\rm d}x\\
  	&\leq \frac{1}{32} \|V_{xtt}(t)\|^{2}+C(1+t)^{-5}\|V_{xt}(t)\|^{2}+C(\varepsilon +\delta)(1+t)^{-4}\|V_{xxt}(t)\|^{2}-\frac{1}{2}\int_{0}^{\infty}\left(g^\prime f\right)_{x}V_{xtt}^{2}{\rm d}x\\
  	&\leq \frac{1}{16} \|V_{xtt}(t)\|^{2}+C(1+t)^{-5}\|V_{xt}(t)\|^{2}+C(\varepsilon +\delta)(1+t)^{-4}\|V_{xxt}(t)\|^{2},	
  	\end{split}
  \end{equation}
  \begin{equation}\label{3.85}
  \begin{split}
  	I_{23}&=\int_{0}^{\infty}\left(gf^\prime\right)_{xt}V_{xx}V_{xtt}{\rm d}x+\int_{0}^{\infty}\left(gf^\prime\right)_{x}V_{xxt}V_{xtt}{\rm d}x-\int_{0}^{\infty}\left(gf^\prime\right)_{t}V_{xxx}V_{xtt}{\rm d}x\\
  	&~~~~+\int_{0}^{\infty}gf^\prime V_{xxxt}V_{xtt}{\rm d}x\\
  	&\leq \frac{1}{32} \|V_{xtt}(t)\|^{2}+C(1+t)^{-7}\|V_{xx}(t)\|^{2}+C(\varepsilon +\delta)(1+t)^{-5}\|V_{xxt}(t)\|^{2}\\
  	&~~~~+C(1+t)^{-6}\|V_{xxx}(t)\|^{2}-\frac{1}{2}\frac{{\rm d}}{{\rm d}t}\int_{0}^{\infty}gf^\prime V_{xxt}^{2}{\rm d}x+\frac{1}{2}\int_{0}^{\infty}\left(gf^\prime\right)_{t}V_{xxt}^{2}{\rm d}x\\
  	&\leq \frac{1}{16} \|V_{xtt}(t)\|^{2}+C(1+t)^{-7}\|V_{xx}(t)\|^{2}+C(\varepsilon +\delta)(1+t)^{-3}\|V_{xxt}(t)\|^{2}\\
  	&~~~~+C(1+t)^{-6}\|V_{xxx}(t)\|^{2}-\frac{1}{2}\frac{{\rm d}}{{\rm d}t}\int_{0}^{\infty}gf^\prime V_{xxt}^{2}{\rm d}x,
  	\end{split}	
  \end{equation}
  and
  \begin{equation}\label{3.86}
  	I_{24}+I_{25}\leq \frac{1}{16} \|V_{xtt}(t)\|^{2}+C\delta(1+t)^{-\frac{17}{2}}+C(\varepsilon +\delta)(1+t)^{-3}\|V_{xxt}(t)\|^{2}.
  \end{equation}
Substituting \eqref{3.81}--\eqref{3.86} into \eqref{3.80}, we have
\begin{equation}\label{3.87}
\begin{split}
    &\frac{1}{2}\frac{{\rm d}}{{\rm d}t}\int_{0}^{\infty}\left[V_{xtt}^{2}+\left(gf^\prime-p^\prime(\bar{v})\right)V_{xxt}^{2}\right]{\rm d}x+\frac{3}{4}\int_{0}^{\infty}V_{xtt}^{2}{\rm d}x \\
    \leq& \frac{1}{2}\frac{{\rm d}}{{\rm d}t}\int_{0}^{\infty}\left[p^\prime(V_{x}+\bar{v}+\hat{v})-p^\prime(\bar{v})\right]V_{xxt}^{2}{\rm d}x+C(1+t)^{-\frac{3}{2}}\|V_{xxt}(t)\|^{2}+C\delta (1+t)^{-\frac{13}{2}}\\&
    ~~~~+C(1+t)^{-4}\|V_{xx}(t)\|^{2}+C(1+t)^{-5}\|V_{x}(t)\|^{2}+C(1+t)^{-3}\|V_{xxx}(t)\|^{2}\\&
    ~~~~+C(1+t)^{-3}\|V_{xt}(t)\|^{2}.
    \end{split}
\end{equation}
Integrating $(1+t)^{k} \cdot \eqref{3.87}$ over $[0,t]$, $k=0,1,2,3,4,$ and using Lemmas \ref{l3.2}--\ref{l3.5}, we can immediately obtain \eqref{3.79}. We have completed the proof of this lemma.
\end{proof}	

\begin{lemma}\label{l3.7}
 If $\varepsilon+\delta \ll 1$, then
 \begin{equation}\label{3.88}
 \begin{split}
   (1+t)^{4}\|V_{tt}(t)\|^{2}&+(1+t)^{5}(\|V_{ttt}(t)\|^{2}+\|V_{xtt}(t)\|^{2})+\int_{0}^{t}(1+s)^{5}\|V_{ttt}(s)\|^{2}{\rm d}s\\
   \leq&  C\left(\|V_{0}\|^{2}_{3}+\|z_{0}\|_{2}^{2}+ \delta \right),
   \end{split}
  \end{equation}
  for $0 \leq t \leq T$.
\end{lemma}
\begin{proof}
Firstly, we have from $\int_{\mathbb{R}^{+}}V_{tt}\times \eqref{1.16}_{1tt}{\rm d}x$ that
\begin{equation}\label{3.89}
  \begin{split}
  \frac{1}{2}\frac{{\rm d}}{{\rm d}t}&\int_{0}^{\infty} \left(V_{tt}^{2}+2V_{tt}V_{ttt}\right){\rm d} x-\int_{0}^{\infty} p^\prime(\bar{v})V_{xtt}^{2}{\rm d} x\\
  &=\int_{0}^{\infty}V_{ttt}^{2}{\rm d}x+\int_{0}^{\infty}[F_{1tt}V_{tt}+(p^\prime(\bar{v})_{tt}V_{x}+2p^\prime(\bar{v})_{t}V_{xt})V_{xtt}]{\rm d}x+\int_{0}^{\infty}F_{2tt}V_{tt}{\rm d}x,
  \end{split}
  \end{equation}
  It is easy to obtain that
\begin{equation}\label{3.90}
 \begin{split}
 &\int_{0}^{\infty}[F_{1tt}V_{tt}+(p^\prime(\bar{v})_{tt}V_{x}+2p^\prime(\bar{v})_{t}V_{xt})V_{xtt}]{\rm d}x\\
 \leq& -\frac{p^\prime(\bar{v})}{16}\|V_{xtt}(t)\|^{2}+C(1+t)^{-5}\|V_{x}(t)\|^{2}+C(1+t)^{-3}\|V_{xt}(t)\|^{2}+C\delta(1+t)^{-\frac{13}{2}},
 \end{split}
\end{equation}
and
\begin{align}\label{3.91}
 \int_{0}^{\infty}F_{2tt}V_{tt}{\rm d}x=&\int_{0}^{\infty}\left(g^\prime fV_{xt}\right)_{tt}V_{tt}{\rm d}x+\int_{0}^{\infty}\left(gf^\prime V_{xx}\right)_{tt}V_{tt}{\rm d}x+\int_{0}^{\infty}(g^\prime f(-p(\bar{v})_{xx}+\hat{v}_{t}))_{tt}V_{tt}{\rm d}x\nonumber\\
&+\int_{0}^{\infty}(gf^\prime(\bar{v}_{x}+\hat{v}_{x}))_{tt}V_{tt}{\rm d}x\nonumber\\
\leq& -\frac{p^\prime(\bar{v})}{16}\|V_{xtt}(t)\|^{2}+C(1+t)^{-2}\|V_{tt}(t)\|^{2}+C(1+t)^{-4}\|V_{xt}(t)\|^{2}+C\|V_{ttt}(t)\|^{2}\nonumber\\
&+C(1+t)^{-6}\|V_{xx}(t)\|^{2}+C(1+t)^{-4}\|V_{xxt}(t)\|^{2}+C\delta(1+t)^{-\frac{13}{2}}.
\end{align}
Putting \eqref{3.90}--\eqref{3.91} into \eqref{3.89}, one gets
\begin{equation}\label{3.92}
  \begin{split}
  &\frac{1}{2}\frac{{\rm d}}{{\rm d}t}\int_{0}^{\infty} \left(V_{tt}^{2}+2V_{tt}V_{ttt}\right){\rm d} x-\frac{3}{4}\int_{0}^{\infty} p^\prime(\bar{v})V_{xtt}^{2}{\rm d} x\\
  \leq&C(1+t)^{-2}\|V_{tt}(t)\|^{2}+C(1+t)^{-3}\|V_{xt}(t)\|_{1}^{2}+C\|V_{ttt}(t)\|^{2}+C(1+t)^{-5}\|V_{x}(t)\|_{1}^{2}\\
&+C\delta(1+t)^{-\frac{13}{2}}.
  \end{split}
  \end{equation}
Next, by calculating $\int_{\mathbb{R}^{+}}\partial_{tt}\eqref{1.16}_{1}\times V_{ttt}{\rm d}x$, we have
\begin{equation}\label{3.93}
\begin{split}
  &\frac{1}{2}\frac{{\rm d}}{{\rm d}t}\int_{0}^{\infty}(V_{ttt}^{2}-p^\prime(\bar{v})V_{xtt}^{2}){\rm d}x+\int_{0}^{\infty}V_{ttt}^{2}{\rm d}x\\
  & = -\int_{0}^{\infty}p^\prime(\bar{v})_{t}V_{xtt}^{2}{\rm d}x+\int_{0}^{\infty}\left[F_{1tt}-\left(p^\prime(\bar{v})_{tt}V_{x}+2p^\prime(\bar{v})_{t}V_{xt}\right)_{x}\right]V_{ttt}{\rm d}x\\
  &~~~~+\int_{0}^{\infty}F_{2tt}V_{ttt}{\rm d}x .
  \end{split}
  \end{equation}
  By using Lemmas \ref{l2.1}--\ref{l2.3} and \eqref{3.4}--\eqref{3.5}, we can conclude that
  \begin{equation}\label{3.94}
    -\int_{0}^{\infty}p^\prime(\bar{v})_{t}V_{xtt}^{2}{\rm d}x \leq C\delta (1+t)^{-\frac{3}{2}}\|V_{xtt}(t)\|^{2},
  \end{equation}
  and
  \begin{equation}\label{3.95}
  \begin{split}
   &\int_{0}^{\infty}\left[F_{1tt}-\left(p^\prime(\bar{v})_{tt}V_{x}+2p^\prime(\bar{v})_{t}V_{xt}\right)_{x}\right]V_{ttt}{\rm d}x \\
   \leq& \frac{1}{16}\|V_{ttt}(t)\|^{2}+\frac{1}{2}\frac{{\rm d}}{{\rm d}t}\int_{0}^{\infty}[p^\prime(V_{x}+\bar{v}+\hat{v})-p^\prime(\bar{v})]V_{xtt}^{2}{\rm d}x+C\delta (1+t)^{-\frac{15}{2}}+C(1+t)^{-4}\|V_{xt}(t)\|^{2}\\&
   +C(1+t)^{-\frac{3}{2}}\|V_{xtt}(t)\|^{2}+C(1+t)^{-3}\|V_{xxt}(t)\|^{2}+C(1+t)^{-6}\|V_{x}(t)\|^{2}\\
   &+C(1+t)^{-5}\|V_{xx}(t)\|^{2}.
   \end{split}
  \end{equation}
  Similar calculations to \eqref{3.51}--\eqref{3.54} deduce that
  \begin{equation}\label{3.96}
 \begin{split}
    &\int_{0}^{\infty}F_{2tt}V_{ttt}{\rm d}x \\&
     \leq \frac{1}{16} \|V_{ttt}(t)\|^{2}-\frac{1}{2}\frac{{\rm d}}{{\rm d}t}\int_{0}^{\infty}gf^\prime V_{xtt}^{2}{\rm d}x+C(1+t)^{-8}\|V_{xx}(t)\|^{2}+C\delta (1+t)^{-\frac{17}{2}}\\&
    ~~~~+C(1+t)^{-5}\|V_{xxt}(t)\|^{2}+C(1+t)^{-3}\|V_{xtt}(t)\|^{2}+C(1+t)^{-6}\|V_{xt}(t)\|^{2}.
    \end{split}
\end{equation}
Substituting \eqref{3.94}--\eqref{3.96} into \eqref{3.93}, we obtain
\begin{equation}\label{3.97}
\begin{split}
   &\frac{1}{2}\frac{{\rm d}}{{\rm d}t}\int_{0}^{\infty}\left[V_{ttt}^{2}+\left(gf^\prime-p^\prime(\bar{v})\right)V_{xtt}^{2}\right]{\rm d}x+\frac{3}{4}\int_{0}^{\infty}V_{ttt}^{2}{\rm d}x \\
    \leq& \frac{1}{2}\frac{{\rm d}}{{\rm d}t}\int_{0}^{\infty}\left[p^\prime(V_{x}+\bar{v}+\hat{v})-p^\prime(\bar{v})\right]V_{xtt}^{2}{\rm d}x +C(1+t)^{-\frac{3}{2}}\|V_{xtt}(t)\|^{2}+C\delta (1+t)^{-\frac{15}{2}}\\&
    ~~~~+C(1+t)^{-5}\|V_{xx}(t)\|^{2}+C(1+t)^{-6}\|V_{x}(t)\|^{2}+C(1+t)^{-4}\|V_{xt}(t)\|^{2}\\&
    ~~~~+C(1+t)^{-3}\|V_{xxt}(t)\|^{2}.
    \end{split}
\end{equation}
 Addition of $\lambda \cdot \eqref{3.92}$ to \eqref{3.97} $(0<\lambda\ll 1)$, we have
  \begin{equation}\label{3.98}
   \begin{split}
   &\frac{1}{2}\frac{{\rm d}}{{\rm d}t}\int_{0}^{\infty} \left[V_{ttt}^{2}+\lambda V_{tt}^{2}+2\lambda V_{tt}V_{ttt}+\left(gf^\prime-p^\prime(\bar{v})\right)V_{xtt}^{2}\right]{\rm d}x+\frac{1}{2}\int_{0}^{\infty} \left(V_{ttt}^{2}-\lambda p^\prime(\bar{v})V_{xtt}^{2}\right){\rm d} x \\
   \leq& \frac{1}{2}\frac{{\rm d}}{{\rm d}t}\int_{0}^{\infty}\left[p^\prime(V_{x}+\bar{v}+\hat{v})-p^\prime(\bar{v})\right]V_{xtt}^{2}{\rm d}x+ C\delta (1+t)^{-\frac{13}{2}}+C(1+t)^{-2}\|V_{tt}(t)\|^{2}\\&
   +C(1+t)^{-\frac{3}{2}}\|V_{xtt}(t)\|^{2}+C(1+t)^{-3}\|V_{xt}(t)\|_{1}^{2}+C(1+t)^{-5}\|V_{x}(t)\|_{1}^{2}.
    \end{split}
    \end{equation}
Multiplying \eqref{3.98} by $(1+t)^{k}, k=0,1,2,3,4,$ and integrating it over $(0,t)$, one can immediately obtain
\begin{equation}\label{3.99}
 \begin{split}
   (1+t)^{4}\|V_{tt}(t)\|_{1}^{2}&+(1+t)^{4}\|V_{ttt}(t)\|^{2}+\int_{0}^{t}(1+s)^{4}(\|V_{xtt}(s)\|^{2}+\|V_{ttt}(s)\|^{2}){\rm d}s\\
   \leq&  C\left(\|V_{0}\|^{2}_{3}+\|z_{0}\|_{2}^{2}+ \delta \right) .
   \end{split}
  \end{equation}
  Integrating $(1+t)^{5} \cdot \eqref{3.97}$ over $(0,t)$ yields
  \begin{equation}\label{3.100}
   (1+t)^{5}(\|V_{xtt}(t)\|^{2}+\|V_{ttt}(t)\|^{2})+\int_{0}^{t}(1+s)^{5}\|V_{ttt}(s)\|^{2}{\rm d}s\leq  C\left(\|V_{0}\|^{2}_{3}+\|z_{0}\|_{2}^{2}+ \delta \right).
  \end{equation}
  \end{proof}
  From \eqref{3.99} and \eqref{3.100}, we can immediately get \eqref{3.88}, the proof of Lemma \ref{l3.7} is completed. By combination of Lemmas \ref{l3.2}--\ref{l3.7}, it is easy to see that {\it a priori }assumption \eqref{3.1} has been closed. Thus we have completed the proof of Proposition \ref{p1}, and obtain \eqref{1.23}--\eqref{1.25}.

\subsection{The proof of Theorem \ref{Thm2}}\label{s3.2}

In this subsection, we will devote ourselves to the proof of Theorem \ref{Thm2}. Firstly, we rewrite \eqref{1.16} to the linearized parabolic problem around $v_{+}$
\begin{equation}\label{3.101}
 \left\{\begin{array}{l}
V_{t}+p^\prime(v_{+})V_{xx}=-V_{tt}+F_{1}+F_{2}+\left[(p^\prime(v_{+})-p^\prime(\bar{v}))V_{x}\right]_{x},\\[2mm]
(V,V_{t})|_{t=0}=(V_0,z_0)(x),\\[2mm]
V|_{x=0}=0.
 \end{array}
        \right.
\end{equation}
Then we have the explicit formula of $V(x,t)$:
\begin{align}\label{3.102}
V(x,t)=&\int_{0}^{\infty}G\left(x,t;y\right) V_{0}\left(y\right) {\rm d}y-\int_{0}^{t} \int_{0}^{\infty}G(x,t-s;y) V_{ss}(y,s) {\rm d}y {\rm d}s \nonumber\\
&+\int_{0}^{t} \int_{0}^{\infty}G(x,t-s;y)\left[(p^\prime(v_{+})-p^\prime(\bar{v}))V_{y}\right]_{y}(y,s) {\rm d}y {\rm d}s\nonumber\\
&+\int_{0}^{t} \int_{0}^{\infty}G(x,t-s;y)(F_{1}+F_{2})(y,s) {\rm d}y {\rm d}s,
\end{align}
where
\begin{equation}\notag
G(x,t;y)=\frac{1}{\sqrt{-4\pi p^\prime(v_{+})t}}\left({\rm e}^{(x-y)^{2}/4p^\prime(v_{+})t}-{\rm e}^{(x+y)^{2}/4p^\prime(v_{+})t}\right).
\end{equation}
By integration by parts in $s$ as in \cite{Nishihara-Yang1999,Nishihara1997},
\begin{align}\label{3.103}
&-\int_{0}^{\frac{t}{2}} \int_{0}^{\infty}G(x,t-s;y) V_{ss}(y,\tau) {\rm d}y {\rm d}s \nonumber\\
=&-\int_{0}^{\infty}G(x,t-s;y) V_{s}(y,s) {\rm d}y\big|_{s=0}^{s=\frac{t}{2}}-\int_{0}^{\frac{t}{2}} \int_{0}^{\infty}G_{t}(x,t-s;y) V_{s}(y,s) {\rm d}y {\rm d}s \nonumber\\
=&\int_{0}^{\infty}G\left(x,t;y\right) z_{0}(y){\rm d}y-\int_{0}^{\infty}G\left(x,t/2;y\right)V_{t}\left(y,t/2\right){\rm d}y\nonumber\\
&-\int_{0}^{\frac{t}{2}} \int_{0}^{\infty}G_{t}(x,t-s;y) V_{s}(y,s) {\rm d}y {\rm d}s.
\end{align}
Therefore, \eqref{3.102} can be rewritten as
\begin{align}\label{3.104}
V(x,t)=&\int_{0}^{\infty}G\left(x,t;y\right) (V_{0}+z_{0})\left(y\right) {\rm d}y-\int_{0}^{\infty}G\left(x,t/2;y\right)V_{t}\left(y,t/2\right) {\rm d}y \nonumber\\
 &-\int_{0}^{\frac{t}{2}} \int_{0}^{\infty}G_{t}(x,t-s;y) V_{s}(y,s) {\rm d}y {\rm d}s-\int_{\frac{t}{2}}^{t} \int_{0}^{\infty}G(x,t-s;y) V_{ss}(y,s) {\rm d}y {\rm d}s \nonumber\\
 &+\int_{0}^{t} \int_{0}^{\infty}G(x,t-s;y)\left[(p^\prime(v_{+})-p^\prime(\bar{v}))V_{y}\right]_{y}(y,s) {\rm d}y {\rm d}s\nonumber\\
 &+\int_{\frac{t}{2}}^{t} \int_{0}^{\infty}G(x,t-s;y)F_{1}(y,s) {\rm d}y {\rm d}s+\int_{\frac{t}{2}}^{t} \int_{0}^{\infty}G(x,t-s;y)F_{2}(y,s) {\rm d}y {\rm d}s\nonumber\\
 &+\int_{0}^{\frac{t}{2}} \int_{0}^{\infty}G(x,t-s;y)F_{1}(y,s) {\rm d}y {\rm d}s+\int_{0}^{\frac{t}{2}} \int_{0}^{\infty}G(x,t-s;y)F_{2}(y,s) {\rm d}y {\rm d}s\nonumber\\
 :=&\sum_{i=1}^{9} J_{i}(x,t).
\end{align}
From Proposition \ref{p1} and \eqref{3.104}, we can first deduce the following lemma.
  \begin{lemma}\label{l3.8}
 Under the assumptions of Theorem \ref{Thm1}, and we assume further that $(V_{0}+z_{0})(x) \in L^{1}$, then $V(x, t)$ satisfies the following improved decay estimates
 \begin{equation}\label{3.105}
\|V(t)\| \leq C(1+t)^{-\frac{1}{4}}.
  \end{equation}
       \end{lemma}
\begin{proof}
 Notice that
\begin{equation}\label{3.106}
\|\partial_{x}^{k}\partial_{t}^{j}G(t)\|_{L^{p}(\mathbb{R})} \leq Ct^{-\frac{1}{2}(1-\frac{1}{p})-\frac{k}{2}-j},\qquad 1\leq p\leq \infty,~k,j \geq 0,
\end{equation}
where $G(x,t)=\frac{1}{\sqrt{-4\pi p^\prime(v_{+})t}}{\rm e}^{\frac{x^{2}}{4p^\prime(v_{+})t}}$.  From \eqref{2.2}--\eqref{2.4}, \eqref{3.2}--\eqref{3.3}, \eqref{3.106} and Hausdorff-Young's inequality, we can deduce that
\begin{equation}\label{3.107}
\left\|J_{1}(t)\right\|\leq \left\|G\left(t\right)\right\|\left\|(V_{0}+z_{0})\right\|_{L^{1}}\leq Ct^{-\frac{1}{4}},
\end{equation}
\begin{equation}\label{3.108}
\|J_{2}(t)\|\leq \left\|G\left(t/2\right)\right\|_{L^{1}} \left\|V_{t}\left(t/2\right)\right\|\leq Ct^{-1},
\end{equation}
\begin{align}\label{3.109}
\|J_{3}(t)\|& \leq \int_{0}^{\frac{t}{2}} \left\|G_{t}(t-s)\right\|_{L^{1}}\left\| V_{s}(s)\right\| {\rm d}s \nonumber\\
&\leq C\int_{0}^{\frac{t}{2}}(t-s)^{-1}(1+s)^{-1}{\rm d}s \leq Ct^{-1}\ln(1+t),
\end{align}
\begin{equation}\label{3.110}
\|J_{4}(t)\|\leq C\int_{\frac{t}{2}}^{t} \left\|G(t-s)\right\|_{L^{1}}\left\|V_{ss}(s)\right\| {\rm d}s \leq C\int_{\frac{t}{2}}^{t}(1+s)^{-2}{\rm d}s \leq Ct^{-1},
\end{equation}
\begin{equation}\label{3.111}
\begin{split}
\|J_{5}(t)\| &\leq C\int_{0}^{t}\left\|G_{y}(t-s)\right\|_{L^{1}}\left\|[(p^\prime(v_{+})-p^\prime(\bar{v}))V_{y}](s)\right\|{\rm d}s\\
&\leq C\int_{0}^{t}\left\|G_{y}(t-s)\right\|_{L^{1}}\left\|V_{y}(s)\right\|(1+s)^{-\frac{1}{2}}{\rm d}s\\
&\leq C\int_{0}^{t}\left\|G_{y}(t-s)\right\|_{L^{1}}(1+s)^{-1}{\rm d}s  \\
&\leq C\left(\int_{0}^{\frac{t}{2}}+\int_{\frac{t}{2}}^{t}\right)(t-s)^{-\frac{1}{2}}(1+s)^{-1}{\rm d}s\leq Ct^{-\frac{1}{2}}\ln(1+t),
\end{split}
\end{equation}
\begin{equation}\label{3.112}
\begin{split}
\|J_{6}(t)\| &\leq\int_{\frac{t}{2}}^{t}\left\|G_{y}(t-s)\right\|(\left\|V_{x}(s)\right\|^{2}+\|\bar{v}_{s}(s)\|_{L^{1}}+\|\hat{v}(s)\|_{L^{1}}){\rm d}s \\
&\leq C\int_{\frac{t}{2}}^{t}(t-s)^{-\frac{3}{4}}(1+s)^{-1}{\rm d}s\leq Ct^{-\frac{3}{4}},
\end{split}
\end{equation}
\begin{equation}\label{3.113}
\|J_{7}(t)\|\leq \int_{\frac{t}{2}}^{t}\left\|G(t-s)\right\|_{L^{1}}\left\|F_{2}(s)\right\|{\rm d}s \leq C\int_{\frac{t}{2}}^{t}(1+s)^{-\frac{9}{4}}{\rm d}s \leq Ct^{-\frac{5}{4}},
\end{equation}
\begin{equation}\notag
 \begin{split}
 J_{8}=&\int_{0}^{\frac{t}{2}}\int_{0}^{\infty}G(x,t-s;y)_{y}\{\left[p(V_{y}+\bar{v}+\hat{v})-p(\bar{v})-p^\prime(\bar{v})V_{y}\right]-p(\bar{v})_{s}\}{\rm d}y {\rm d}s \\
 =&\int_{0}^{\frac{t}{2}}\int_{0}^{\infty}G(x,t-s;y)_{y}\left[p(V_{y}+\bar{v}+\hat{v})-p(\bar{v})-p^\prime(\bar{v})V_{y}\right]{\rm d}y {\rm d}s\\
 &+\int_{0}^{\infty}\{G(x,t;y)_{y}[p(\bar{v}(y,0))-p(v_{+})]-G(x,t/2;y)_{y}[p(\bar{v}(y,t/2))-p(v_{+})]\}{\rm d}y\\
 &+\int_{0}^{\frac{t}{2}}\int_{0}^{\infty}G(x,t-s;y)_{yt}[p(\bar{v})-p(v_{+})](y,s) {\rm d}y{\rm d}s.
  \end{split}
\end{equation}
Applying Lemmas \ref{l2.1}--\ref{l2.3} and \eqref{3.2}--\eqref{3.3}, we obtain
\begin{equation}\label{3.114}
  \begin{split}
\|J_{8}(t)\| \leq&\int_{0}^{\frac{t}{2}}\left\|G_{y}(t-s)\right\|(\left\|V_{x}(s)\right\|^{2}+\|\hat{v}(s)\|_{L^{1}}){\rm d}s+\left\|G_{x}(t)\right\|\\
&+\int_{0}^{\frac{t}{2}}\left\|G_{yt}(t-s)\right\|\|(\bar{v}-v_{+})(s)\|_{L^{1}}){\rm d}s\leq Ct^{-\frac{3}{4}}.
  \end{split}
\end{equation}

We turn to deal with $J_{9}$ now, notice that
\begin{align}\label{3.115}
J_{9}&=\int_{0}^{\frac{t}{2}}\int_{0}^{\infty}G(x,t-s;y)F_{2}(y,s) {\rm d}y{\rm d} \nonumber\\
&=\int_{0}^{\frac{t}{2}}\int_{0}^{\infty} G(x,t-s;y)[g^\prime f(V_{ys}+\bar{u}_{y}+\hat{v}_{s})+gf^\prime(V_{yy}+\bar{v}_{y}+\hat{v}_{y})](y,s) {\rm d}y {\rm d}s.
\end{align}
Since $\hat{u}(x,t)$ doesn't belong to any $L^{p}$ space for $1 \leq p < \infty$, it means that $J_{8}$ is estimated quite differently from $J_{5}$. It is easy to show that
\begin{align}\label{3.116}
&\left\|\int_{0}^{\frac{t}{2}}\int_{0}^{\infty}G(x,t-s;y)(g^\prime fV_{ys})(y,s) {\rm d}y {\rm d}s\right\| \nonumber\\
\leq& \left\|\int_{0}^{\frac{t}{2}}\int_{0}^{\infty}G_{y}(x,t-s,y)(g^\prime fV_{s})(y,s) {\rm d}y {\rm d}s \right\|+\left\|\int_{0}^{\frac{t}{2}}\int_{0}^{\infty}G(x,t-s,y)[V_{s}(g^{\prime } f)_{y}](y,s) {\rm d}y {\rm d}s\right\| \nonumber\\
\leq&C\int_{0}^{\frac{t}{2}}\left(\left\|G_{y}(t-s)\right\|_{L^{1}}\left\|(g^\prime fV_{s})(s)\right\|+\left\|G(t-s)\right\|\|V_{s}(s)\|\|(g^{\prime } f)_{y}(s)\|\right){\rm d}s \nonumber\\
\leq& Ct^{-\frac{1}{2}}\int_{0}^{\frac{t}{2}}(1+s)^{-2}{\rm d}s +Ct^{-\frac{1}{4}}\int_{0}^{\frac{t}{2}}(1+s)^{-\frac{9}{4}}{\rm d}s\leq Ct^{-\frac{1}{4}},
\end{align}
and
\begin{equation}\label{3.117}
\begin{split}
&\left\|\int_{0}^{\frac{t}{2}}\int_{0}^{\infty} G(x,t-s;y)[g^\prime f(\bar{u}_{y}+\hat{v}_{s})+gf^\prime(\bar{v}_{y}+\hat{v}_{y})](y,s) {\rm d}y {\rm d}\tau\right\|\\
&~~~~~\leq \int_{0}^{\frac{t}{2}}\left\|G(t-s)\right\|\left\|[g^\prime f(\bar{u}_{y}+\hat{v}_{s})+gf^\prime(\bar{v}_{y}+\hat{v}_{y})](\tau)\right\|_{L^{1}} {\rm d}s\\
&~~~~~\leq C\int_{0}^{\frac{t}{2}}(t-s)^{-\frac{1}{4}}(1+s)^{-2}{\rm d}s \leq Ct^{-\frac{1}{4}}.
\end{split}
\end{equation}
Without any difficulty, we can similarly proof
\begin{equation}\label{3.118}
\left\|\int_{0}^{\frac{t}{2}}\int_{0}^{\infty}G(x,t-s;y)(gf^\prime V_{yy})(y,s) {\rm d}y {\rm d}s \right\| \leq Ct^{-\frac{1}{4}}.
\end{equation}
To sum up, we obtain
\begin{equation}\label{3.119}
\|V(t)\| \leq \sum_{i=1}^{9}\|J_{i}(t)\|\leq C(1+t)^{-\frac{1}{4}}.
\end{equation}
Thus, we have completed the proof of Lemma \ref{l3.8}.
\end{proof}
With the above preparations in hand, we now turn to give the following lemma.
\begin{lemma}\label{l3.9}
 Under the conditions in Lemma \ref{l3.8}, then $V(x, t)$ satisfies the following decay rates:
  \begin{align}
&\label{3.120}\|\partial_{x}^{k}\partial_{t}^{j}V(t)\| \leq C(1+t)^{-\frac{1}{4}-\frac{k}{2}-j},\quad 0\leq k+j \leq 3,~0\leq j \leq 2,\\
&\label{3.121}\|\partial_{t}^{3}V(t)\|\leq C(1+t)^{-\frac{11}{4}}.
\end{align}
       \end{lemma}
 \begin{proof}
  Firstly, integrating $(1+t)^{\epsilon_{0}+\frac{1}{2}}\times \eqref{3.23}$ with respect to $t$ over $(0, t)$ for any fixed $0<\epsilon_{0}<\frac{1}{2}$, one gets
\begin{align}\notag
  &\frac{1}{2}(1+t)^{\epsilon_{0}+\frac{1}{2}}\int_{0}^{\infty} \left(V_{t}^{2}+\lambda V^{2}+2\lambda VV_{t}-p^\prime(\bar{v})V_{x}^{2}\right){\rm d} x\nonumber\\
  &+\frac{1}{2}\int_{0}^{t}\int_{0}^{\infty} (1+s)^{\epsilon_{0}+\frac{1}{2}}\left(V_{t}^{2}-\lambda p^\prime(\bar{v}) V_{x}^{2} \right){\rm d} x{\rm d}\tau \nonumber\\
  \leq&C\int_{0}^{t}(1+s)^{\epsilon_{0}-\frac{1}{2}}(\|V(s)\|_{1}^{2}+\|V_{t}(s)\|^{2}){\rm d}s+C(\varepsilon +\delta)(1+t)^{\epsilon_{0}+\frac{1}{2}}\|V_{x}(t)\|^{2}\nonumber\\
  &+C(\|V_{0}\|_{2}^{2}+\|V_{1}\|_{1}^{2}+\delta).\nonumber
  \end{align}
  Combining \eqref{3.2}--\eqref{3.3} with \eqref{3.105} shows that
  \begin{align}\notag
  \int_{0}^{t}(1+s)^{\epsilon_{0}-\frac{1}{2}}(\|V(s)\|_{1}^{2}+\|V_{t}(s)\|^{2}){\rm d}s\leq C\int_{0}^{t}(1+s)^{\epsilon_{0}-1}{\rm d}s\leq C(1+t)^{\epsilon_{0}}. \nonumber
  \end{align}
  Noticing that $\varepsilon$ and $\delta$ are sufficiently small, then one can immediately obtain
  \begin{align}\label{3.122}
 (1+t)^{\epsilon_{0}+\frac{1}{2}}(\|V(t)\|_{1}^{2}+\|V_{t}(t)\|^{2})+\int_{0}^{t}(1+s)^{\epsilon_{0}+\frac{1}{2}}\left(\|V_{x}(s)\|^{2}+\|V_{t}(s)\|^{2}\right){\rm d}s \leq C(1+t)^{\epsilon_{0}}.
  \end{align}
 Next, integrating of $(1+t)^{\epsilon_{0}+\frac{3}{2}}\times \eqref{3.22}$ over $(0, t)$, we have
  \begin{align}\notag
  &\frac{1}{2}(1+t)^{\epsilon_{0}+\frac{3}{2}}\int_{0}^{\infty}\left(V_{t}^{2}-p^\prime(\bar{v})V_{x}^{2}\right){\rm d} x+\frac{7}{8}\int_{0}^{t}\int_{0}^{\infty}(1+s)^{\epsilon_{0}+\frac{3}{2}} V_{t}^{2}{\rm d} x{\rm d}s\nonumber\\
  \leq &C\int_{0}^{t}(1+s)^{\epsilon_{0}+\frac{1}{2}}(\|V_{x}(s)\|^{2}+\|V_{t}(s)\|^{2}){\rm d}s+C(\varepsilon +\delta)(1+t)^{\epsilon_{0}+\frac{3}{2}}\|V_{x}(t)\|^{2}\nonumber\\
  &+C(\|V_{0}\|_{2}^{2}+\|V_{1}\|_{1}^{2}+\delta).\nonumber
  \end{align}
  Combining the above two equations, we have
  \begin{equation}\label{3.123}
  (1+t)^{\epsilon_{0}+\frac{3}{2}}(\|V_{x}(t)\|^{2}+\|V_{t}(t)\|^{2})+\int_{0}^{t}(1+s)^{\epsilon_{0}+\frac{3}{2}}\|V_{t}(s)\|^{2}{\rm d}s\leq C(1+t)^{\epsilon_{0}}.
  \end{equation}
  Similarly, by integrating $(1+t)^{\epsilon_{0}+\frac{3}{2}}\times \eqref{3.36}$ over $(0, t)$, we can get
  \begin{align}\notag
  &\frac{1}{2}(1+t)^{\epsilon_{0}+\frac{3}{2}}\int_{0}^{\infty} \left[V_{xt}^{2}+\lambda V_{x}^{2}+2\lambda V_{xt}V_{x}+\left(-p^\prime(\bar{v})+gf^\prime\right)V_{xx}^{2}\right]{\rm d} x\nonumber\\
  &+\frac{1}{2}\int_{0}^{t}\int_{0}^{\infty}(1+s)^{\epsilon_{0}+\frac{3}{2}}\left(V_{xt}^{2}-\lambda p^\prime(\bar{v})V_{xx}^{2}\right){\rm d}x{\rm d}s \nonumber\\
   \leq& C\int_{0}^{t}(1+s)^{\epsilon_{0}+\frac{1}{2}}(\|V_{x}(s)\|_{1}^{2}+\|V_{xt}(s)\|^{2}){\rm d}s+C(\varepsilon +\delta)(1+t)^{\epsilon_{0}+\frac{3}{2}}\|V_{xx}(t)\|^{2}\nonumber\\
  &+C(\|V_{0}\|_{2}^{2}+\|V_{1}\|_{1}^{2}+\delta).\nonumber
  \end{align}
  We can obtain from \eqref{3.2}--\eqref{3.3} and \eqref{3.122} that
  \begin{equation}\notag
  \int_{0}^{t}(1+s)^{\epsilon_{0}+\frac{1}{2}}(\|V_{x}(s)\|_{1}^{2}+\|V_{xt}(s)\|^{2}){\rm d}s \leq C(1+t)^{\epsilon_{0}},
  \end{equation}
  then it follows that
\begin{equation}\label{3.124}
  (1+t)^{\epsilon_{0}+\frac{3}{2}}(\|V_{x}(t)\|_{1}^{2}+\|V_{xt}(t)\|^{2})+\int_{0}^{t}(1+s)^{\epsilon_{0}+\frac{3}{2}}(\|V_{xx}(s)\|^{2}+\|V_{xt}(s)\|^{2}){\rm d}s \leq C(1+t)^{\epsilon_{0}}.
  \end{equation}
  Integrating $(1+t)^{\epsilon_{0}+\frac{5}{2}}\times \eqref{3.35}$ over $(0, t)$, we have
  \begin{align}\notag
  &\frac{1}{2}(1+t)^{\epsilon_{0}+\frac{5}{2}}\int_{0}^{\infty}\left(V_{xt}^{2}+\left(-p^\prime(\bar{v})+gf^\prime\right)V_{xx}^{2}\right){\rm d} x+\frac{7}{8}\int_{0}^{t}\int_{0}^{\infty}(1+s)^{\epsilon_{0}+\frac{5}{2}} V_{xt}^{2}{\rm d} x{\rm d}s\nonumber\\
  \leq &C\int_{0}^{t}(1+s)^{\epsilon_{0}+\frac{3}{2}}(\|V_{xx}(s)\|^{2}+\|V_{xt}(s)\|^{2}){\rm d}s+C(\varepsilon +\delta)(1+t)^{\epsilon_{0}+\frac{5}{2}}\|V_{xx}(t)\|^{2}\nonumber\\
  &+C(\|V_{0}\|_{2}^{2}+\|V_{1}\|_{1}^{2}+\delta).\nonumber
  \end{align}
  By employing \eqref{3.124}, we can immediately obtain
\begin{equation}\label{3.125}
  (1+t)^{\epsilon_{0}+\frac{5}{2}}(\|V_{xx}(t)\|^{2}+\|V_{xt}(t)\|^{2})+\int_{0}^{t}(1+s)^{\epsilon_{0}+\frac{5}{2}}\|V_{xt}(s)\|^{2}{\rm d}s \leq C(1+t)^{\epsilon_{0}}.
  \end{equation}
 In a similar process as above, integrating $(1+t)^{\epsilon_{0}+\frac{5}{2}}\times \eqref{3.56}$ and $(1+t)^{\epsilon_{0}+\frac{7}{2}}\times \eqref{3.55}$ over $(0, t)$, we obtain
 \begin{equation}\label{3.126}
  (1+t)^{\epsilon_{0}+\frac{5}{2}}(\|V_{xx}(t)\|_{1}^{2}+\|V_{xxt}(t)\|^{2})+\int_{0}^{t}(1+s)^{\epsilon_{0}+\frac{5}{2}}(\|V_{xxx}(s)\|^{2}+\|V_{xxt}(s)\|^{2}){\rm d}s \leq C(1+t)^{\epsilon_{0}},
  \end{equation}
\begin{equation}\label{3.127}
  (1+t)^{\epsilon_{0}+\frac{7}{2}}(\|V_{xxx}(t)\|^{2}+\|V_{xxt}(t)\|^{2})+\int_{0}^{t}(1+s)^{\epsilon_{0}+\frac{7}{2}}\|V_{xxt}(s)\|^{2}{\rm d}s \leq C(1+t)^{\epsilon_{0}}.
  \end{equation}
By integrating $(1+t)^{\epsilon_{0}+\frac{5}{2}}\times \eqref{3.76}$ and $(1+t)^{\epsilon_{0}+\frac{7}{2}}\times \eqref{3.75}$ over $(0, t)$, one gets
\begin{equation}\label{3.128}
  (1+t)^{\epsilon_{0}+\frac{5}{2}}(\|V_{t}(t)\|_{1}^{2}+\|V_{tt}(t)\|^{2})+\int_{0}^{t}(1+s)^{\epsilon_{0}+\frac{5}{2}}(\|V_{xt}(s)\|^{2}+\|V_{tt}(s)\|^{2}){\rm d}s \leq C(1+t)^{\epsilon_{0}},
  \end{equation}
\begin{equation}\label{3.129}
  (1+t)^{\epsilon_{0}+\frac{7}{2}}(\|V_{xt}(t)\|^{2}+\|V_{tt}(t)\|^{2})+\int_{0}^{t}(1+s)^{\epsilon_{0}+\frac{7}{2}}\|V_{tt}(s)\|^{2}{\rm d}s \leq C(1+t)^{\epsilon_{0}}.
  \end{equation}
By integrating $(1+t)^{\epsilon_{0}+\frac{9}{2}}\times \eqref{3.87}$ over $(0, t)$, we get
\begin{equation}\label{3.130}
  (1+t)^{\epsilon_{0}+\frac{9}{2}}(\|V_{xxt}(t)\|^{2}+\|V_{xtt}(t)\|^{2})+\int_{0}^{t}(1+s)^{\epsilon_{0}+\frac{9}{2}}\|V_{xtt}(s)\|^{2}{\rm d}s \leq C(1+t)^{\epsilon_{0}}.
  \end{equation}
  By integrating $(1+t)^{\epsilon_{0}+\frac{9}{2}}\times \eqref{3.98}$ and $(1+t)^{\epsilon_{0}+\frac{11}{2}}\times \eqref{3.97}$ over $(0, t)$, we get
  \begin{equation}\label{3.131}
  (1+t)^{\epsilon_{0}+\frac{9}{2}}(\|V_{tt}(t)\|_{1}^{2}+\|V_{ttt}(t)\|^{2})+\int_{0}^{t}(1+s)^{\epsilon_{0}+\frac{9}{2}}(\|V_{xtt}(s)\|^{2}+\|V_{ttt}(s)\|^{2}){\rm d}s \leq C(1+t)^{\epsilon_{0}},
  \end{equation}
  and
\begin{equation}\label{3.132}
  (1+t)^{\epsilon_{0}+\frac{11}{2}}(\|V_{xtt}(t)\|^{2}+\|V_{ttt}(t)\|^{2})+\int_{0}^{t}(1+s)^{\epsilon_{0}+\frac{11}{2}}\|V_{ttt}(s)\|^{2}{\rm d}s \leq C(1+t)^{\epsilon_{0}}.
  \end{equation}
Hence, by combination of \eqref{3.122}--\eqref{3.132}, we can immediately obtain the desired estimates \eqref{3.120}--\eqref{3.121}. the proof of Lemma \ref{l3.9} is completed.
\end{proof}

Based on Lemma \ref{l3.9}, we can derive the following Lemma.

\begin{lemma}\label{l3.10}
In addition to the assumptions stated in Lemma \ref{l3.9}, we assume futher that $u_{+}=0$, $\int_{0}^{\infty}(V_{0}+\frac{1}{\alpha}z_{0})(x){\rm d}x=0$ and $W_{0}(x)\in L^{1}$, then the following improved decay rates are ture
  \begin{align}
  &\label{3.133}\|\partial_{x}^{k}\partial_{t}^{j}V(t)\| \leq C(1+t)^{-\frac{3}{4}-\frac{k}{2}-j},\quad 0\leq k+j \leq 3,~0\leq j \leq 2,  \\
  &\label{3.134}\|\partial_{t}^{3}V(t)\|\leq C(1+t)^{-\frac{13}{4}}.
      \end{align}
\end{lemma}
\begin{proof}
We firstly proof \eqref{3.133} in the case of $0\leq k\leq1$. Define
\begin{equation}\label{3.135}
M(t):= \sup \limits_{0 \leq s \leq t,~0 \leq k \leq 1}(1+s)^{\frac{3}{4}+\frac{k}{2}}\|\partial_{x}^{k}V(s)\|.
  \end{equation}
Now we ready to show $M(t)$ is bounded. One can easily deduce that
 \begin{equation}\label{3.136}
\left\|J_{1}(t)\right\|\leq \left\|G_{y}\left(t\right)\right\|\left\|W_{0}\right\|_{L^{1}} \leq t^{-\frac{3}{4}}.
\end{equation}
Notice that
\begin{equation}\label{3.137}
\begin{split}
 J_{5}=&\int_{0}^{t}\int_{0}^{\infty}G(x,t-s;y)\left[(p^\prime(v_{+})-p^\prime(\bar{v}))V_{y}\right]_{y}(y,s) {\rm d}y {\rm d}s\\
 =&\left(-\int_{0}^{\frac{t}{2}}-\int_{\frac{t}{2}}^{t}\right)\int_{0}^{\infty}G_{y}(x,t-s;y)[(p^\prime(v_{+})-p^\prime(\bar{v}))V_{y}](y,s) {\rm d}y {\rm d}s:=J_{5}^{1}+J_{5}^{2}.
 \end{split}
 \end{equation}
 By using \eqref{2.2}, \eqref{3.120} and \eqref{3.135}, we have
 \begin{equation}\label{3.138}
 \begin{split}
 \left\|J_{5}^{1}(t)\right\|&\leq \int_{0}^{\frac{t}{2}}\left\|G_{y}(t-s)\right\|\|V_{x}(s)\|\|(\bar{v}-v_{+})(s)\|{\rm d}s\\
 &\leq C\delta M(t)\int_{0}^{\frac{t}{2}}(t-s)^{-\frac{3}{4}}(1+s)^{-\frac{3}{2}}{\rm d}s\leq C\delta M(t)t^{-\frac{3}{4}},
 \end{split}
  \end{equation}
  and
  \begin{equation}\label{3.139}
    \begin{split}
   \left\|J_{5}^{2}(t)\right\|&\leq \int_{\frac{t}{2}}^{t}\left\|G_{y}(t-s)\right\|_{L^{1}}\|V_{x}(s)\|\|(\bar{v}-v_{+})(s)\|_{L^{\infty}}{\rm d}s\\
   &\leq \int_{\frac{t}{2}}^{t}(t-s)^{-\frac{1}{2}}(1+s)^{-\frac{5}{4}}{\rm d}s\leq Ct^{-\frac{3}{4}}.
     \end{split}
   \end{equation}
Furthermore, since $u_{+}=0$, then it follows that $(\hat{v},\hat{u}) \equiv (0,0),$ in view of \eqref{1.22}, $\eqref{1.6}_{2}$ and \eqref{2.2}, we get
 \begin{align}\label{3.140}
\left\|J_{9}(t)\right\|&=\left\|\int_{0}^{\frac{t}{2}}\int_{0}^{\infty} G(x,t-s;y)_{y}(gf)(y,s) {\rm d}y {\rm d}s\right\|\leq \int_{0}^{\frac{t}{2}}\left\|G_{y}(t-s)\right\|\left\|(gf)(s)\right\|_{L^{1}}{\rm d}s\nonumber\\
& \leq C\int_{0}^{\frac{t}{2}}\left\|G_{y}(t-s)\right\|\left\|(V_{s}+\bar{u})(s)\right\|^{2}{\rm d}s \leq Ct^{-\frac{3}{4}}\int_{0}^{\frac{t}{2}}(1+s)^{-\frac{3}{2}}{\rm d}s \leq Ct^{-\frac{3}{4}}.
\end{align}
Consequently,
\begin{equation}\label{3.141}
\|V(t)\| \leq \sum_{i=1}^{9}\|J_{i}(t)\|\leq C(1+\delta M(t))(1+t)^{-\frac{3}{4}}.
\end{equation}
As introduced in the introduction, the condition $u_{+}=0$ is used only in obtaining \eqref{3.140}, the main reason is that $\hat{u}(x,t)$ doesn't belong to any $L^{p}$ space for $1 \leq p < \infty$. For damped compressible Euler equations, this condition can be removed.

With the above preparations in hand, we are going to prove $M(t)$ is bounded.  Firstly, by combintion of \eqref{3.11}--\eqref{3.14} and \eqref{3.120}, it is easy to verify that
\begin{equation}\label{3.141a}
\int_{\mathbb{R}}F_{2}V{\rm d} x\leq -\frac{p^\prime(\bar{v})}{16}\|V_{x}(t)\|^{2} +C\|V_{t}(t)\|^{2}+ C\delta (1+t)^{-\frac{5}{2}}.
\end{equation}
It follows from \eqref{3.8}--\eqref{3.9} and \eqref{3.141a} that
 \begin{align}\label{3.141b}
   \frac{1}{2}\frac{{\rm d}}{{\rm d}t}\int_{0}^{\infty} \left(V^{2}+2VV_{t}\right){\rm d} x-\frac{3}{4}\int_{0}^{\infty} p^\prime(\bar{v})V_{x}^{2}{\rm d}x \leq C\|V_{t}(t)\|^{2}+ C\delta (1+t)^{-\frac{5}{2}}.
 \end{align}
 Addition of $\lambda \cdot \eqref{3.141b}$, $0<\lambda \ll 1$ to \eqref{3.22} yields
 \begin{align}\label{3.141c}
    &\frac{1}{2}\frac{{\rm d}}{{\rm d}t}\int_{0}^{\infty} \left(V_{t}^{2}+\lambda V^{2}+2\lambda VV_{t}-p^\prime(\bar{v})V_{x}^{2}\right){\rm d} x+\frac{1}{2}\int_{0}^{\infty}(V_{t}^{2}-\lambda p^\prime(\bar{v})V_{x}^{2}){\rm d}x\nonumber\\
     \leq&\frac{{\rm d}}{{\rm d}t}\int_{0}^{\infty}\left[\int_{\bar{v}}^{V_{x}+\bar{v}+\hat{v}}p(s){\rm d}s-p(\bar{v})V_{x}-\frac{p^\prime(\bar{v})}{2}V_{x}^{2}\right]{\rm d}x+ C\delta (1+t)^{-\frac{5}{2}}.
  \end{align}
By applying the same idea as the upper lemma, integrating $(1+t)^{\epsilon_{0}+\frac{3}{2}}\times \eqref{3.141c}$ with respect to $t$ over $(0, t)$ for any fixed $0<\epsilon_{0}<\frac{1}{2}$, we have
\begin{align}\notag
  &\frac{1}{2}(1+t)^{\epsilon_{0}+\frac{3}{2}}\int_{0}^{\infty} \left(V_{t}^{2}+\lambda V^{2}+2\lambda VV_{t}-p^\prime(\bar{v})V_{x}^{2}\right){\rm d} x\nonumber\\
  &+\frac{1}{2}\int_{0}^{t}\int_{0}^{\infty} (1+s)^{\epsilon_{0}+\frac{3}{2}}\left(V_{t}^{2}-\lambda p^\prime(\bar{v}) V_{x}^{2} \right){\rm d} x{\rm d}\tau \nonumber\\
  \leq&C\int_{0}^{t}(1+s)^{\epsilon_{0}+\frac{1}{2}}(\|V(s)\|_{1}^{2}+\|V_{t}(s)\|^{2}){\rm d}s+C(\varepsilon +\delta)(1+t)^{\epsilon_{0}+\frac{3}{2}}\|V_{x}(t)\|^{2}+C(1+t)^{\epsilon_{0}}.\nonumber
  \end{align}
  We have from \eqref{3.120}--\eqref{3.121} with \eqref{3.141} that
  \begin{align}\notag
  \int_{0}^{t}&(1+s)^{\epsilon_{0}+\frac{1}{2}}(\|V(s)\|_{1}^{2}+\|V_{t}(s)\|^{2}){\rm d}s\nonumber\\
  &\leq C(1+\delta^{2}M^{2}(t))\int_{0}^{t}(1+s)^{\epsilon_{0}-1}{\rm d}s\leq C(1+\delta^{2}M^{2}(t))(1+t)^{\epsilon_{0}}. \nonumber
  \end{align}
  Noticing $\varepsilon +\delta\ll 1$, then combining the above two equations yields
  \begin{align}\label{3.142}
 (1+t)^{\epsilon_{0}+\frac{3}{2}}&(\|V(t)\|_{1}^{2}+\|V_{t}(t)\|^{2})+\int_{0}^{t}(1+s)^{\epsilon_{0}+\frac{3}{2}}\left(\|V_{x}(s)\|^{2}+\|V_{t}(s)\|^{2}\right){\rm d}s\nonumber\\
 & \leq C(1+\delta^{2}M^{2}(t))(1+t)^{\epsilon_{0}}.
  \end{align}
 Next, integrating of $(1+t)^{\epsilon_{0}+\frac{5}{2}}\times \eqref{3.22}$ over $(0, t)$, we have
  \begin{align}\notag
  &\frac{1}{2}(1+t)^{\epsilon_{0}+\frac{5}{2}}\int_{0}^{\infty}\left(V_{t}^{2}-p^\prime(\bar{v})V_{x}^{2}\right){\rm d} x+\frac{3}{4}\int_{0}^{t}\int_{0}^{\infty}(1+s)^{\epsilon_{0}+\frac{5}{2}} V_{t}^{2}{\rm d} x{\rm d}s\nonumber\\
  \leq &C\int_{0}^{t}(1+s)^{\epsilon_{0}+\frac{3}{2}}(\|V_{x}(s)\|^{2}+\|V_{t}(s)\|^{2}){\rm d}s+C(\varepsilon +\delta)(1+t)^{\epsilon_{0}+\frac{5}{2}}\|V_{x}(t)\|^{2}\nonumber\\
  &+C(1+t)^{\epsilon_{0}}.\nonumber
  \end{align}
  Combining the above two equations, we have
  \begin{equation}\label{3.143}
  (1+t)^{\epsilon_{0}+\frac{5}{2}}(\|V_{x}(t)\|^{2}+\|V_{t}(t)\|^{2})+\int_{0}^{t}(1+s)^{\epsilon_{0}+\frac{5}{2}}\|V_{t}(s)\|^{2}{\rm d}s\leq C(1+\delta^{2}M^{2}(t))(1+t)^{\epsilon_{0}}.
  \end{equation}
It follows from \eqref{3.142}--\eqref{3.143} that
\begin{equation}\notag
\sum \limits_{k=0}^{1}(1+t)^{\frac{3}{2}+k}\|\partial_{x}^{k}V(t)\|^{2} \leq C+C\delta^{2}M^{2}(t).
\end{equation}
Thus, we can immediately get
\begin{equation}\notag
M^{2}(t)\leq C+C\delta^{2}M^{2}(t).
\end{equation}
Since $\delta$ is sufficiently small, one gets
\begin{equation}\notag
M(t)\leq C,
\end{equation}
which implies
\begin{equation}\label{3.144}
\|\partial_{x}^{k}V(t)\| \leq C(1+t)^{-\frac{3}{4}-\frac{k}{2}},\quad 0\leq k\leq 1.
\end{equation}
Once we have obtained \eqref{3.144}, in the similar way in Lemma \ref{l3.9}, we can obtain
\begin{align}\label{3.145}
\sum \limits_{0\leq k+j\leq 3,0\leq j\leq 2}&(1+t)^{\epsilon_{0}+\frac{3}{2}+k+2j}\|\partial_{x}^{k}\partial_{t}^{j}V(t)\|^{2}+(1+t)^{\epsilon_{0}+\frac{13}{2}}\|\partial_{t}^{3}V(t)\|^{2} \leq C(1+t)^{\epsilon_{0}}.
\end{align}
Then one can immediately obtain \eqref{3.133}--\eqref{3.134}. This completes the proof of Lemma \ref{l3.10}.
\end{proof}
It is easy to see that Theorem \ref{Thm2} follows from Lemma \ref{l3.9}--Lemma \ref{l3.10}.

\vspace{6mm}

 \noindent {\bf Acknowledgements:} The research was supported by the National Natural Science Foundation of China $\#$12171160, 11771150, 11831003 and Guangdong Basic and Applied Basic Research Foundation $\#$2020B1515310015.

\bigbreak

%
%
%
{\small

\bibliographystyle{plain}

}

\end{document}